\documentclass[11pt,dvipsnames]{article} 
\usepackage{hyperref}
\usepackage[margin = 1in]{geometry}
\usepackage{setspace}%\onehalfspace
\usepackage[utf8]{inputenc}
\usepackage{amsmath,amssymb,amsfonts, mathbbol}
\usepackage{mathrsfs,latexsym,amsthm}
\usepackage{tikz}\usetikzlibrary{arrows} \tikzset{->, node distance =
  15mm, >=stealth', shorten > = 1pt, ArrowNode/.style =
  {font=\footnotesize}} \usepackage{tikz-cd} \usepackage[all]{xy}
\usepackage{enumitem}
\setlist[itemize]{label=$\circ$, leftmargin = *}
\setlist[enumerate]{label=$(\alph*)$, leftmargin = *}
\theoremstyle{plain}
\newtheorem{lemma}{Lemma}[section]
\newtheorem{proposition}[lemma]{Proposition}
\newtheorem{corollary}[lemma]{Corollary}
\newtheorem{definition}[lemma]{Definition}
\newtheorem{theorem}[lemma]{Theorem}

\theoremstyle{definition}

\newtheorem{example}[lemma]{Example}

\newcommand{\lra}{\leftrightarrows}
\newcommand{\ra}{\rightarrow}

\newcommand{\inclu}{\hookrightarrow}

\newcommand{\up}{{\uparrow}}
\newcommand{\ca}{\mathcal}

\newcommand{\se}{\subseteq}
\newcommand{\sm}{\setminus}
\newcommand{\we}{\wedge}

\newcommand{\bve}{\bigvee}

\newcommand{\bca}{\bigcap}
\newcommand{\bcu}{\bigcup}
\newcommand{\bd}[1]{\mathbf{#1}}

\newcommand{\na}{\nabla}
\newcommand{\del}{\Delta}

\newcommand{\id}{\mathsf{id}}

\newcommand{\Om}{\Omega}

\newcommand{\pt}{\mathsf{pt}}

\newcommand{\cohfrith}{{\bf CohFrith}}
\newcommand{\cohfrm}{{\bf CohFrm}}
\newcommand{\Frm}{{\bf Frm}}

\newcommand{\ffrm}{{\bf Frith}}
\newcommand{\sffrm}{{\bf Frith}_{\rm sym}}
\newcommand{\qunifrm}{{\bf QUniFrm}}
\newcommand{\unifrm}{{\bf UniFrm}}

\newcommand{\dlat}{{\bf DLat}}
\newcommand{\perv}{{\bf Pervin}}
\newcommand{\sperv}{{\bf Pervin}_{\rm sym}}
\newcommand{\Top}{{\bf Top}}
\newcommand{\psym}{{\bf Sym}_{\rm
    Perv}} % symmetrization for Pervin spaces
\newcommand{\fsym}{{\bf Sym}_{\rm
    Frith}} % symmetrization for Frith frames
\newcommand{\sym}{{\bf Sym}_{\rm QUniFrm}} % symmetrization for
\newcommand{\cB}{{\mathcal B}} \newcommand{\cC}{{\mathcal C}}
\newcommand{\cD}{{\mathcal D}} \newcommand{\cE}{{\mathcal E}}
\newcommand{\cF}{{\mathcal F}} \newcommand{\cJ}{{\mathcal J}}
\newcommand{\cL}{{\mathcal L}} 
\newcommand{\cP}{{\mathcal P}} \newcommand{\cS}{{\mathcal S}}
\newcommand{\cT}{{\mathcal T}} \newcommand{\cU}{{\mathcal U}}

\newcommand{\adj}{{\;\leftrightarrows\;}}  \newcommand{\idl}{{\rm Idl}} 
 \newcommand{\two}{{\bf 2}}
 \newcommand{\luni}{{\vartriangleleft}}

\newcommand{\just}[2]{\stackrel{#1}{#2}}

\begin{document}
\title{A pointfree theory of Pervin spaces\footnote{Célia Borlido was
    partially supported by the Centre for Mathematics of the
    University of Coimbra - UIDB/00324/2020, funded by the Portuguese
    Government through FCT/MCTES. Anna Laura Suarez received funding
    from the European Research Council (ERC) under the European
    Union's Horizon 2020 research and innovation program (grant
    agreement No.670624).}} \author{C\'{e}lia Borlido\footnote{Centre
    for Mathematics of the University of Coimbra (CMUC),
    \mbox{3001-501}~Coimbra, Portugal, cborlido@mat.uc.pt} \and Anna
  Laura Suarez\footnote{Laboratoire J. A. Dieudonn\'e UMR CNRS 7351,
    Universit\'e Nice Sophia Antipolis, 06108~Nice~Cedex~02,
    France, annalaurasuarez993@gmail.com}}\maketitle
%%%%%%%%%%%%%%%%%%%%%%%%%%%%%%%%%%%%%%%%%%%%%%%%%%%%%%%%%%%%%%%%%%%%%%
\begin{abstract}
  We lay down the foundations for a pointfree theory of Pervin
  spaces. A Pervin space is a set equipped with a bounded sublattice
  of its powerset, and it is known that these objects characterize
  those quasi-uniform spaces that are transitive and totally
  bounded. The pointfree notion of a Pervin space, which we call Frith
  frame, consists of a frame equipped with a generating bounded
  sublattice.
  In this paper we introduce and study the category of Frith frames
  and show that the classical dual adjunction between topological
  spaces and frames extends to a dual adjunction between Pervin spaces
  and Frith frames.
  Unlike what happens for Pervin spaces, we do not have an equivalence
  between the categories of transitive and totally bounded
  quasi-uniform frames and of Frith frames, but we show that the
  latter is a full coreflective subcategory of the former.
  We also explore the notion of completeness of Frith frames inherited
  from quasi-uniform frames, providing a characterization of those
  Frith frames that are complete and a description of the completion
  of an arbitrary Frith frame.

\end{abstract}
%%%%%%%%%%%%%%%%%%%%%%%%%%%%%%%%%%%%%%%%%%%%%%%%%%%%%%%%%%%%%%%%%%%%%%

\section{Introduction}

\emph{Pervin quasi-uniformities} were introduced to answer the
question of whether a given topology is induced by some
quasi-uniformity.  Pervin~\cite{PERVIN1962} showed that, if $\tau$ is
a topology on $X$, then~$\tau$ is the topology induced by the
quasi-uniformity generated by the \emph{entourages} of the form
\[E_U := (X \times U) \cup (U^c \times X),\]
for $U \in
\tau$. Abstracting this idea, we may consider on a set $X$ the
quasi-uniformity $\ca{E}_{\cF}$ generated by $\{E_U \mid U \in \cF\}$,
for an arbitrary family $\cF \subseteq \cP(X)$, and the quasi-uniform
spaces obtained in this way are known as \emph{Pervin
  spaces}. In~\cite[Proposition~2.1]{FletcherLindgren1982} it is shown
that the quasi-uniform spaces $(X, \cE_\tau)$ when $\tau$ is a
topology on $X$ are transitive and totally bounded, but the same
argument works for every space $(X, \cE_\cF)$. Actually, more can be
said: \emph{every} transitive and totally bounded quasi-uniformity on
a set $X$ is of the form $\cE_\cF$ for some family $\cF \subseteq
\cP(X)$. To the best of our knowledge, this result is due to Gehrke,
Grigorieff, and Pin in unpublished work (see also~\cite{pin17}), and
at the present date, a proof can be found in the Goubault-Larrecq's
blog~\cite{Goubault-Larrecq-blog}. Another interesting aspect of
Pervin spaces is that the bounded sublattice of $\cP(X)$ generated by
some family $\cF$ can be recovered from the quasi-uniform space $(X,
\cE_\cF)$: it consists of the subsets $U \subseteq \cP(X)$ such that
$E_U \in \cE_\cF$ (see~\cite[Theorem~5.1]{GehrkeGool2014} for a
proof). For that reason, Pervin spaces may be elegantly represented by
pairs $(X, \cS)$, where $X$ is a set and $\cS$ is a bounded sublattice
of $\cP(X)$.

The main contribution of this paper is the development of a pointfree
theory of Pervin spaces. The central object of study are the pairs of
the form $(L, S)$, where $L$ is a frame and $S$ is a join-dense
bounded sublattice. We name such pairs of \emph{Frith frames}. This
choice is justified by the fact that the pointfree version of Pervin's
construction is known as the \emph{Frith
  quasi-uniformity}~\cite{Frith86} on a congruence frame. The
correctness of our notion is evidenced by the existence of a
Pervin-Frith dual adjunction extending the classical dual adjunction
between topological spaces and frames
(cf. Proposition~\ref{p:24}). One major difference in the pointfree
setting is the lack of an equivalence between the categories of Frith
frames and of transitive and totally bounded quasi-uniform frames,
although Frith frames do form a full coreflective subcategory of
transitive and totally bounded quasi-uniform frames
(cf. Theorem~\ref{t:4}). The picture is different when we restrict to
\emph{symmetric Frith frames} (that is, those Frith frames $(L,S)$
where $S$ is a Boolean algebra) on the one hand, and to transitive and
totally bounded uniform frames on the other
(cf. Corollary~\ref{c:5}). We show that every Frith frame admits a
\emph{symmetrization}, which defines a reflection of Frith frames onto
the symmetric ones (cf. Proposition~\ref{p:25}). Moreover, the
symmetrization of a Frith frame is shown to be, on the one hand, a
restriction of the usual uniform reflection of quasi-uniform frames
(cf. Proposition~\ref{p:2}) and, on the other hand, a pointfree
version of the \emph{symmetrization} of a Pervin space
(cf. Proposition~\ref{p:27} and Theorem~\ref{t:5}).  Finally, we
explore the notion of \emph{complete} Frith frame, which is naturally
inherited from the homonymous notion for quasi-uniform frames
(cf. Proposition~\ref{p:1}), both from the point of view of
\emph{dense extremal epimorphisms} and of \emph{Cauchy maps}. In
particular, we characterize the complete Frith frames as those whose
frame component is \emph{coherent}
(cf. Theorem~\ref{manycharacterizations}).

The paper is organized as follows.  Section~\ref{sec:prelim} is a
preliminary section, where we present the background needed and
establish the notation used in the rest of the paper. In
Section~\ref{sec:pervin} we give an overview of the theory of Pervin
spaces. While we did not intend to go very deep in our exposition, we
tried to provide enough details to allow the reader to compare the
known results with our pointfree approach. In Section~\ref{sec:1} we
introduce the category of Frith frames and discuss some of its general
properties. In particular, we show the existence of a dual adjunction
between Pervin spaces and Frith frames (Section~\ref{FrFrm}), we
discuss compactness, coherence and zero-dimensionality of Frith frames
(Section~\ref{sec:coh}), we show that the category of Frith frames is
complete and cocomplete (Section~\ref{sec:lim-colim}), and we
characterize some special morphisms (Sections~\ref{sec:spec-mor}
and~\ref{sec:td}).  In Section~\ref{sec:2} we show how to assign a
transitive and totally bounded quasi-uniform frame to each Frith frame
and show that this assignment defines a full coreflective
embedding. In Section~\ref{sec:4} we consider the special case of
symmetric Frith frames and of (transitive and totally bounded) uniform
frames. In particular, we show that the category of symmetric Frith
frames is equivalent to the category of transitive and totally bounded
uniform frames, and we show that the symmetric Frith frames form a
full reflective and coreflective subcategory of the category of Frith
frames. Section~\ref{sec:7} is devoted to the characterization of
complete Frith frames and of the completion of a Frith frame.  More
precisely, in Section~\ref{sec:8} we discuss completion using
\emph{dense extremal epimorphisms}, while in Section~\ref{sec:9} we
characterize complete Frith frames using \emph{Cauchy maps}.
  Finally, in Section~\ref{sec:10}, we discuss possible connections
  with existing literature on \emph{frames with bases}.

%%%%%%%%%%%%%%%%%%%%%%%%%%%%%%%%%%%%%%%%%%%%%%%%%%%%%%%%%%%%%%%%%%%%%%
\section{Preliminaries}\label{sec:prelim}
In this section we state the basic results and set up the
notation that will be used in the rest of the paper.  The reader is
assumed to have basic knowledge of category theory, as well as some
acquaintance with general topics from pointfree topology.
For more on category theory, the reader is referred
to~\cite{MacLane1998}. For further reading on frame theory, including
uniform and quasi-uniform frames, we refer
to~\cite{picadopultr2011frames}.  Although not strictly needed, some
background on quasi-uniform spaces may also be useful. For this topic,
we refer to~\cite{FletcherLindgren1982}.
%%%%%%%%%%%%%%%%%%%%%%%%%%%%%%%%%%%%%%%%%%%%%%%%%%%%%%%%%%%%%%%%%%%%%% 
\subsection{The category of frames and the $\Top$ - $\Frm$ dual
  adjunction}

A \emph{frame} is a complete lattice~$(L, \bigvee, \bigwedge, 0, 1)$
such that for every $a \in L$ and $\{b_i\}_{i \in I} \subseteq L$ the
following equality holds:
\[a \wedge \bigvee_{i \in I} b_i = \bigvee_{i \in I}(a \wedge b_i).\]
A \emph{frame homomorphism} is a map $h: L \to M$ that preserves
finite meets and arbitrary joins (including empty meets and empty
joins). % A frame homomorphism need not preserve the pseudocomplement
% operation, but whenever $a$ is complemented, we have that $h(a)$ is
% complemented, too, and $h(a^*) = h(a)^*$.
%
We denote by $\Frm$ the category of frames and frame homomorphisms.
It is well-known that in $\Frm$ the one-to-one homomorphisms are
precisely the monomorphisms, while the onto homomorphisms are the
extremal epimorphisms.\footnote{In a category~$\cC$, an
  epimorphism~$e$ is \emph{extremal} if whenever $e=m\circ g$ with $m$
  a monomorphism we have that $m$ is an isomorphism.} Extremal
epimorphisms of frames will be simply called \emph{surjections}.
A frame homomorphism~$h: L \to M$ is \emph{dense} provided $h(a) = 0$
implies $a = 0$. %  We already observed that frame homomorphisms need
% not preserve the pseudocomplement operation. However, we have the
% following:

% \begin{lemma}\label{l:4}
%   Let $h: L \to M$ be a dense extremal epimorphism. Then, for every $a \in L$,
%   we have $h(a^*) = h(a)^*$.
% \end{lemma}

The frame distributive law guarantees that, for every element $a \in
L$, there is a greatest element $a^*$, called the
\emph{pseudocomplement} of $a$, satisfying $a \wedge a^* = 0$. That
is, $a^*$ is defined by the following property:
\begin{equation}
  \forall x \in L, \ x \wedge a = 0 \iff x \leq a^*.\label{eq:5}
\end{equation}
When $a \vee a^* = 1$, we have that $a$ is \emph{complemented} and
$a^*$ is the \emph{complement} of~$a$. % Moreover, if $a$ is
% complemented then $a^{**} = a$.

Because the two finitary distributive laws are duals of each other,
every frame is a bounded distributive lattice.  The category of
bounded distributive lattices and lattice homomorphisms will be
denoted by~$\dlat$. We will only consider lattices that are
distributive and have empty meets and empty joins (that is, a top and a
bottom element). Therefore we will sometimes omit the words
``bounded'' and ``distributive'' when referring to a lattice.  For a
frame $L$ and a subset $R\se L$, we denote as $\langle R
\rangle_\dlat$ the sublattice of $L$ generated by $R$. The expression
$\langle R \rangle_\Frm$ denotes the subframe of $L$ generated by $R$.

Finally, we denote by \Top{} the category of topological spaces and
continuous functions. At the base of pointfree topology is the
classical adjunction $\Omega: \Top \adj \Frm^{\rm op}: \pt$ between
the category of topological spaces and (the dual of the) category of
frames~\cite{johnstone82, picadopultr2011frames}. Points of a frame
$L$ will be seen as frame homomorphisms $p:L \to \two$ where $\two :=
\{0 < 1\}$ denotes the two-element lattice. The basic open sets that
generate the topology on~$\pt(L)$ will be denoted by
\[\widehat a := \{p \in \pt(L) \mid p(a) = 1\},\] for each $a \in
L$.% For the
\subsection{The congruence frame}
Frame surjections may be characterized, up to isomorphism, via frame
congruences: if $h: L \twoheadrightarrow M$ is a surjection, then its
kernel $\ker (h) := \{(x, y) \mid h(x) = h(y)\}$ is a frame
congruence, and conversely, every frame congruence~$\theta$ induces a
frame surjection $L \twoheadrightarrow L/{\theta}$. These two
assignments are mutually inverse. Congruences are closed under
arbitrary intersections, and because of this they form a closure
system within the poset of all binary relations on a frame. For a
frame $L$, we say that the congruence \emph{generated} by a relation
$\rho\se L\times L$ is the smallest congruence containing $\rho$, and
we denote it by $\overline{\rho}$. The set $\cC L$ of all frame
congruences of a frame~$L$ is naturally ordered by inclusion, which
endows $\cC L$ with a frame structure given by
\[\bigwedge_{i \in I} \theta_i = \bigcap_{i \in I}\theta_i \qquad
  \text{and}\qquad \bve_{i \in I}\theta_i \,= \,\overline{\bigcup_{i
      \in I}\theta_i} \, =\, \bca \{\theta\in \cC L\mid \bcu_{i \in I}
  \theta_i\se \theta\}\]
for every family of congruences $\{\theta_i\}_{i \in I}$.  There are
two types of distinguished congruences: the so-called closed
ones, and the so-called \emph{open} ones.  These are, respectively,
the congruences~$\nabla_a$ generated by an element of the form $(0,
a)$, and the congruences~$\Delta_a$ generated by an element of the
form $(a, 1)$, for some $a \in L$.  It is not hard to see that
\[\nabla_a = \{(x, y) \mid x \vee a = y \vee a\}\qquad
  \text{and}\qquad\Delta_a = \{(x, y) \mid x \wedge a = y \wedge
  a\}.\]
We have the following basic results about open and closed congruences,
where $\id_L$ denotes the identity congruence $\{(a,a)\mid a\in L\}$
on~$L$.
\begin{lemma}\label{basic} For a frame $L$, and elements $a, b, a_i
  \in L$ ($i \in I$), the following holds:
  \begin{itemize}
  \item $\del_0=L\times L$ and $\del_1=\id_L$;
  \item $\na_0=\id_L$ and $\na_1=L\times L$;
  \item $\Delta_a \vee \Delta_b = \Delta_{a \wedge b}$ and $\bca_{i
      \in I} \del_{a_i}=\del_{\bve_{i \in I} a_i}$;
  \item $\bve_{i\in I} \na_{a_i}=\na_{\bve_{i \in I} a_i}$ and
    $\nabla_a \wedge \nabla_b = \nabla_{a \wedge b}$.
  \end{itemize}
\end{lemma}
Open and closed congruences suffice to generate~$\cC L$, meaning that
if we close the collection $\{\del_a, \na_a\mid a\in L\}$ under finite
meets and the resulting set under arbitrary joins we get the whole
congruence frame $\cC L$.  In this paper, some subframes of~$\cC L$
will be particularly relevant, namely those generated by a set of the
form $\{\nabla_a \mid a \in L\} \cup \{\Delta_s \mid s \in S\}$, and
denoted $\cC_SL$, for some subset $S \subseteq L$.  We note that the
assignment $a \mapsto \nabla_a$ defines a frame embedding $\nabla: L
\hookrightarrow \cC_SL$. For that reason, we will often treat $L$ as a
subframe of $\cC_SL$. The following generalizes a well-known property
of the congruence frame:

\begin{proposition}[{\cite[Theorem~16.2]{Wilson1994TheAT}}]\label{p:10}
  Let $L$ and $M$ be frames and $S\se L$ be a subset. Then, every
  frame homomorphism $h:L\ra M$ such that $h(s)$ is complemented for
  all $s\in S$ may be uniquely extended to a frame homomorphism
  $\widetilde{h}:\cC_S L\ra M$, so that the following diagram
  commutes:
  \begin{center}
    \begin{tikzpicture}
      [->, node distance = 20mm] \node (A) {$L$}; \node[right of = A]
      (B) {$\cC_S L$}; \node[below of = B, yshift = 2mm] (C) {$M$};
      \draw[right hook->] (A) to node[above, ArrowNode] {$\nabla$} (B);
      \draw[dashed] (B) to node[right, ArrowNode] {$\widetilde h$}
      (C); \draw (A) to node[below, ArrowNode] {$h$}(C);
    \end{tikzpicture}
  \end{center}
\end{proposition}

In particular, we have the following:

\begin{corollary}\label{c:1}
  Let $h: L \to M$ be a frame homomorphism, and let $S \subseteq L$
  and $T \subseteq M$ be subsets such that $h[S] \subseteq T$. Then,
  $h$ may be uniquely extended to a frame homomorphism $\overline h:
  \cC_SL \to \cC_TM$. Moreover, for every $a \in L$ and $s \in S$, the
  following equalities hold:
  \[\overline{h}(\nabla_a) = \nabla_{h(a)}\qquad \text{and}\qquad
    \overline{h}(\Delta_s) = \Delta_{h(s)}.\]
\end{corollary}

%%%%%%%%%%%%%%%%%%%%%%%%%%%%%%%%%%%%%%%%%%%%%%%%%%%%%%%%%%%%%%%%%%%%%%
\subsection{Compact, coherent, and zero-dimensional
  frames.}\label{sec:3}

Let $L$ be a frame. Recall that an element $a \in L$ is \emph{compact}
if whenever $a \leq \bigvee_{i \in I}a_i$ there exists a finite subset
$F \subseteq I$ such that $a \leq\bigvee_{i \in F} a_i$, and~$L$ is
\emph{compact} if its top element is compact. % Notice that a
% topological space is compact if and only if its frame of opens is.
We
denote by $K(L)$ the set of compact elements of a frame~$L$.
% Clearly, $K(L)$ is a join-subsemilattice of~$L$. If we further have
% that
If $K(L)$ is a join-dense sublattice of~$L$, then we say that $L$ is
\emph{coherent}. We denote by $\bd{CohFrm}$ the subcategory of $\Frm$
whose objects are the coherent frames, and morphisms are those frame
homomorphisms that preserve compact elements. Coherent frames will be
a central concept in this paper.

Given a bounded distributive lattice $S$, we denote by $\idl(S)$ its
\emph{ideal completion}. We will often see $S$ as a sublattice
of~$\idl(S)$, by identifying an element $s \in S$ with the principal
ideal ${\downarrow}s$ it generates. The assignment $S \mapsto \idl(S)$
is part of a functor $\idl(-): \dlat \to \Frm$ which is an equivalence
of categories when co-restricted to $\cohfrm$.

% A typical example of a coherent frame is the \emph{ideal completion}
% $\idl(S)$ of any bounded distributive lattice~$S$. Indeed, it is easy
% to see that the compact elements of~$\idl(S)$ are precisely the
% principal ideals of~$S$, and these form a sublattice isomorphic
% to~$S$, which is clearly join-dense in~$\idl(S)$.
% We will often see $S$ as a sublattice of~$\idl(S)$, by identifying an
% element $s \in S$ with the principal ideal ${\downarrow}s$ it
% generates. It turns out that every coherent frame is the ideal
% completion of its sublattice of compact
% elements.

\begin{proposition}[{\cite[pages 59 and 65]{Johnstone85}}]\label{p:15}
  The category~$\Frm$ of frames is a reflective subcategory of the
  category~$\dlat$ of bounded distributive lattices. The reflector
  $\idl(-): \dlat \to \Frm$ sends a lattice $S$ to its ideal
  completion $\idl(S)$ and a lattice homomorphism $h: S \to T$ to
  \[\idl(h): \idl(S) \to \idl(T), \ J \mapsto
    \langle{h[J]}\rangle_{\idl} = {\downarrow}h[J].\]
  Moreover, the corestriction of $\idl(-)$ to \cohfrm{} induces an
  equivalence of categories whose inverse $K(-): \cohfrm \to \dlat$
  sends a coherent frame to its sublattice of compact elements, and a
  morphism to its suitable restriction and co-restriction.
\end{proposition}
\begin{corollary}\label{c:7}
  Every lattice homomorphism $h: S \to L$, with $L$ a frame, uniquely
  extends to a frame homomorphism $\widehat h: \idl(S) \to L$ defined
  by $\widehat h(J):= \bigvee h[J]$.
\end{corollary}

By definition of coherent frame homomorphism, it is not hard to verify
that $K(-): \cohfrm \to \dlat$ is both a right and a left adjoint of
$\idl(-): \dlat \to \cohfrm$. By suitably composing the adjunctions
$\cohfrm \adj \dlat$ and $\dlat \adj \Frm$, and using the equivalence
$\cohfrm \cong \dlat$, we obtain the following
(see~\cite[Proposition~1]{banaschewski1981} for a direct proof):
\begin{proposition}\label{p:16}
  The category $\cohfrm$ of coherent frames and coherent frame
  homomorphisms is a coreflective subcategory of the category $\Frm$
  of frames and frame homomorphisms. The coreflector is the
  restriction and co-restriction $\idl(-):\Frm \to \cohfrm$ of the
  functor $\idl(-)$.
\end{proposition}

Finally, we say that a frame $L$ is \emph{zero-dimensional} if the
sublattice $B(L)$ of its complemented elements join-generates the
frame.

\begin{proposition}\label{p:12}
  Let $L$ be a frame. Then,
  \begin{enumerate}
  \item\label{item:14} if $L$ is compact, then every complemented
    element is compact;
  \item\label{item:17} if $L$ is zero-dimensional, then every compact
    element is complemented.
  \end{enumerate}
  In particular, every compact and zero-dimensional frame is coherent.
\end{proposition}
% \begin{proof}
%   For~\ref{item:14}, suppose that $L$ is compact and let $a\in L$ be a
%   complemented element. If $a\leq \bve_{i\in I} a_i$ for some family
%   $\{a_i\}_{i \in I}$, then $1 = a^* \vee a \leq a^*\ve \bve_{i\in I}
%   a_i$. By compactness of $L$ there is a finite subset $F\se I$
%   satisfying $1 \leq a^*\ve \bve_{i\in F}a_i$, and thus, $a = a \wedge
%   1\leq a \wedge ( a^*\ve \bve_{i\in F}a_i) = \bve_{i\in
%     F}a_i$. For~\ref{item:17}, suppose that $L$ is zero-dimensional
%   and let $a\in L$ be a compact element. Since $a$ is a join of
%   complemented elements, by compactness, it is also a finite join of
%   complemented elements. Therefore, it is complemented as well.
% \end{proof}

%%%%%%%%%%%%%%%%%%%%%%%%%%%%%%%%%%%%%%%%%%%%%%%%%%%%%%%%%%%%%%%%%%%%%% 
\subsection{Uniform and quasi-uniform frames}

We will introduce (quasi-)uniformities on a frame~$L$ using
\emph{entourages}. Entourages on~$L$ are nothing but certain elements
of the coproduct $L \oplus L$, whose construction we recall now
  (see~\cite[page 60]{Johnstone85} or~\cite[Chapter~IV,
  Section~5]{picadopultr2011frames}).

\begin{definition}{\footnote{This
      is a slight abuse of terminology because \emph{$C$-ideals} are
      originally defined as depending on a parameter $C$~\cite[page
      58]{Johnstone85}. The notion of \emph{$C$-ideal} adopted here is also known
      as \emph{coproduct ideal} or \emph{cp-ideal} (see, for instance,
      \cite{PicadoPultr2018})}}
  Let $L$ be a frame. \emph{We will call \emph{$C$-ideal} of $L$ to} a
  subset $A \subseteq L \times L$ that satisfies the following
  properties:
  \begin{enumerate}[label = (J.\arabic*)]
  \item\label{item:J1} $A$ is a down-set;
  \item\label{item:J2} if $\{(a_i, b)\}_{i \in I}\subseteq A$, then
    $(\bigvee_{i \in I} a_i, b) \in A$;
  \item\label{item:J3} if $\{(a, b_i)\}_{i \in I} \subseteq A$, then
    $(a, \bigvee_{i \in I} b_i) \in A$.
  \end{enumerate}
\end{definition} Given $(a, b) \in L \times L$, we denote by $a
  \oplus b$ the smallest $C$-ideal containing $(a,b)$, that is, $a
  \oplus b = {\downarrow} (a, b) \cup (L \times \{0\}) \cup (\{0\} \times
  L)$.%  The following properties will be needed later:
% \begin{lemma}\label{l:14}
%   For every $a, a', b, b',a_i, b_i \in L$ ($i \in I$), the following
%   equalities hold:
%   \begin{itemize}
%   \item $(a \oplus b) \wedge (a' \oplus b') = (a \wedge a')\oplus (b
%     \wedge b')$;
%   \item $\bigvee_{i \in I}(a_i \oplus b) = (\bigvee_{i \in I}a_i)
%     \oplus b$;
%   \item $\bigvee_{i \in I}(a \oplus b_i) = a \oplus (\bigvee_{i \in I}
%     b_i)$.
%   \end{itemize}
% \end{lemma}

\begin{proposition}
  The coproduct $L \oplus L$ is isomorphic to the set of all $C$-ideals
  of $L$ ordered by inclusion. The coproduct maps are injections, and
  they are given by
  \[(\iota_1:L \hookrightarrow L \oplus L, \quad a \mapsto a \oplus
    1)\qquad \text{and}\qquad(\iota_2:L \hookrightarrow L \oplus L,
    \quad b \mapsto 1 \oplus b).\]
\end{proposition}
Let $L$ and $M$ be frames. If $f_i:L \to M$ ($i = 1,2$) are frame
homomorphisms, then we denote by $f_1\oplus f_2$ the unique frame
homomorphism $L \oplus L \to M \oplus M$ given by the universal
property of coproducts.
% Explicitly, for each $C$-ideal $A \in L \oplus
% L$, we have
% \begin{equation}
%   (f_1 \oplus f_2) (A) = \bigvee \{f_1(a) \oplus f_2(b) \mid (a, b)
%   \in A\}.\label{eq:2}
% \end{equation}

%%%%%%%%%%%%%%%%%%%%%%%%%%%%%%%%%%%%%%%%%%%%%%%%%%%%%%%%%%%%%%%%%%%%%%
Any pair of $C$-ideals $A, B \in L \oplus L$ may be composed as follows:
\[A \circ B := \bigvee \{a \oplus b \mid \exists c \neq 0 \colon (a,
  c) \in A, \ (c, b) \in B\},\]
and for every set $\cJ \subseteq L \oplus L$ of $C$-ideals, we have
the following relations on~$L$:
\[b \luni^\cJ_1 a \iff A \circ (b \oplus b) \subseteq (a \oplus a)
  \text{ for some $A \in \cJ$},\] and
\[b \luni^\cJ_2 a \iff (b \oplus b) \circ A \subseteq (a \oplus a)
  \text{ for some $A \in \cJ$}.\]
When $\cJ$ is clear from the context, we may simply
write~$\luni_i$ instead of~$\luni_i^\cJ$. Given $i \in \{1,2\}$, we
denote
\[\cL_i(\cJ):= \{a \in L \mid a = \bigvee \{b \mid b \luni_i a\}\}.\] It
is well-known that, if $\cJ$ is a filter basis, then each $\cL_i(\cJ)$
is a subframe of~$L$.

We can now introduce quasi-uniform frames. We follow the approach in
\cite{Picado95} and \cite{Picado952}. A \emph{(Weil) entourage} on~$L$
is a $C$-ideal $E \subseteq L \oplus L$ that satisfies
\[\bigvee \{a \mid (a, a) \in E\} = 1.\]
It can be shown that every entourage $E$ is contained in $E \circ
E$. If the equality $E = E \circ E$ holds, then~$E$ is said to be
\emph{transitive}.  An entourage $E$ is \emph{finite} if it contains
some finite join $\bigvee_{i = 1}^n a_i \oplus a_i$, with $a_1\vee
\dots \vee a_n = 1$. Notice that every finite intersection of
transitive and finite entourages is again transitive and finite. Also
notice that, if $E$ is an entourage, then so is $E^{-1}:= \{(b, a)
\mid (a, b) \in E\}$. An entourage~$E$ is \emph{symmetric} if $E = E^{-1}$.
\begin{definition}
  A \emph{quasi-uniformity} on $L$ is a subset $\cE \subseteq L \oplus
  L$ of entourages on $L$ such that
  \begin{enumerate}[label = (QU.\arabic*)]
  \item\label{item:QU1} $\cE$ is a filter,
  \item\label{item:QU2} for every $E \in \cE$, there exists some $F
    \in \cE$ such that $F \circ F \subseteq E$,
  \item\label{item:QU3} $L$ is the frame generated by $\cL_1(\cE) \cup
    \cL_2(\cE)$.
  \end{enumerate}
  The set $\cE$ is called a \emph{uniformity} if it further satisfies
  \begin{enumerate}[label = (QU.\arabic*)]
  \setcounter{enumi}{3}
  \item\label{item:QU4} if $E \in \cE$, then $E^{-1} \in \cE$.
  \end{enumerate}
  A \emph{(quasi-)uniform frame} is a pair $(L, \cE)$, where $L$ is a
  frame and $\cE$ is a (quasi-)uniformity on~$L$.
\end{definition}
If $\cE' \subseteq L \oplus L$ is a filter (sub)basis of entourages
that satisfies~\ref{item:QU2} and~\ref{item:QU3}, then we say that
$\cE'$ is a \emph{(sub)basis for a quasi-uniformity} on~$L$. In that
case, the filter generated by~$\cE'$ is a quasi-uniformity. If $\cE'$
consists of symmetric entourages, then the quasi-uniformity it
generates is actually a uniformity.  For every quasi-uniform frame
$(L, \cE)$, the set $\{E \cap E^{-1} \mid E \in \cE\}$ is a basis
consisting of symmetric entourages, and it generates the coarsest
uniformity~$\overline \cE$ on~$L$ that contains~$\cE$. Finally, a
quasi-uniform frame~$(L, \cE)$ is \emph{transitive} if it has a basis
of transitive entourages, and \emph{totally bounded} if it has a basis
of finite entourages.

A \emph{(quasi-)uniform homomorphism} $h: (L, \cE) \to (M, \cF)$ is a
frame homomorphism $h: L \to M$ such that for every $E \in \cE$ we
have $(h \oplus h) (E) \in \cF$. We denote by \qunifrm{} the category
of quasi-uniform frames and quasi-uniform homomorphisms, and by
\unifrm{} the category of uniform frames and uniform frame morphisms. The following proposition is proven in \cite{FletcherHunsakerLindgren1994}, in which the authors use an alternative definition of quasi-uniform frame. This definition is shown in \cite{Picado95} to be equivalent to the one we use.

\begin{proposition}
  [{\cite[Corollary 4.7]{FletcherHunsakerLindgren1994}}]\label{p:11}
  Uniform frames form a full reflective subcategory of quasi-uniform
  frames. The reflector $\sym:\qunifrm \to \unifrm$ maps a
  quasi-uniform frame $(L, \cE)$ to $(L, \overline \cE)$, and a
  morphism to itself.
\end{proposition}

A (quasi-)uniform homomorphism $h: (L, \cE) \to (M, \cF)$ is an
extremal epimorphism if and only if it is onto and $\cF$ is the
quasi-uniformity generated by $(h \oplus h)[\cE]$.  A (quasi-)uniform
frame $(M, \cF)$ is said to be \emph{complete} provided every dense
extremal epimorphism~$(L, \cE) \twoheadrightarrow (M, \cF)$, with $(L,
\cE)$ a \mbox{(quasi-)uniform} frame, is an isomorphism. A
\emph{completion} of a (quasi-)uniform frame~$(L, \cE)$ is a complete
(quasi-)uniform frame~$(M, \cF)$ together with a dense extremal
epimorphism $(M, \cF) \twoheadrightarrow (L, \cE)$. The next two
results will be needed in the sequel.

\begin{proposition}[{\cite[Proposition~3.3]{Frith1999}}]\label{l:2}
  A quasi-uniform frame~$(L, \cE)$ is complete if and only if so is
  its uniform reflection~$(L, \overline \cE)$.
\end{proposition}
\begin{proposition}[{\cite[Chapter VII, Proposition
    2.2.2]{picadopultr2011frames}}]\label{p:22}
  Let $h:(L, \cE) \twoheadrightarrow (M, \cF)$ be a dense extremal
  epimorphism of uniform frames. If $M$ is compact, then $h$ is an
  isomorphism.
\end{proposition}

Completions of quasi-uniform frames may also be characterized in terms
of the so-called \emph{Cauchy maps}. These can be thought of as
analogues of Cauchy filters for quasi-uniform spaces.
If~$L$ and~$M$ are semilattices with top and bottom elements, then a
map $\phi: L \to M$ is called a \emph{bounded meet homomorphism}
provided it preserves the bottom element and all finite meets
(including the empty one).
\begin{definition}\label{sec:cm}
  Let $(L, \cE)$ be a quasi-uniform frame and $M$ be any frame. A
  \emph{Cauchy map} $\phi: (L, \cE) \to M$ is a function $\phi: L \to M$ such that
  \begin{enumerate}
  \item\label{item:20} $\phi$ is a bounded meet homomorphism,
  \item\label{item:1} for every $a \in L$, $\phi(a) \leq \bigvee
    \{\phi(b) \mid b \luni_1 a\text{ or } b \luni_2 a\}$,
  \item\label{item:2} for every $E \in \cE$, $1 = \bigvee \{\phi(a)
    \mid (a,a) \in E\}$.
  \end{enumerate}
\end{definition}

\begin{theorem}[{\cite[Theorem 6.5]{PicadoPultr2018}}]\label{t:6}
  Let $(L, \cE)$ be a quasi-uniform frame. Then $(L, \cE)$ is complete
  if and only if each Cauchy map $(L, \cE) \to M$ is a frame
  homomorphism.
\end{theorem}

%%%%%%%%%%%%%%%%%%%%%%%%%%%%%%%%%%%%%%%%%%%%%%%%%%%%%%%%%%%%%%%%%%%%%%
\section{Pervin spaces}\label{sec:pervin}

In what follows, a \emph{Pervin space} will be a pair $(X,\ca{S})$
such that $X$ is a set and $\ca{S}\se \cP(X)$ is a bounded
sublattice. A morphism
of Pervin spaces $f: (X, \cS) \to (Y, \cT)$ is a function $f: X \to Y$
such that whenever $T\in \cT$ we also have $f^{-1}(T)\in \cS$. The
category of Pervin spaces and corresponding morphisms will be denoted
by {\perv}. As mentioned in the introduction, each Pervin space
uniquely determines a \emph{transitive and totally bounded
  quasi-uniform space}, and every such space arises from a Pervin
space. Actually, one can show that there is an equivalence between the
categories of Pervin spaces and of transitive and totally bounded
quasi-uniform spaces. However, since quasi-uniform spaces are not the
subject of this paper, but rather its pointfree counterpart, we will
not provide further details on this matter. We refer the reader
to~\cite{PERVIN1962,FletcherLindgren1982, pin17} for further reading.

We start by noticing that \Top{} may be seen as a full subcategory of
\perv{}, with a topological space $(X, \tau)$ being identified with
the Pervin space~$(X, \tau)$. Conversely, if $(X, \cS)$ is a Pervin
space, then we may consider on $X$ the topology $\Omega_\cS(X)$
generated by~$\cS$.\footnote{The reader interested in the connections
  between Pervin and quasi-uniform spaces may notice that, if $(X,
  \cE_\cS)$ is the quasi-uniform space defined by $(X, \cS)$, then
  $\Omega_\cS(X)$ is precisely the topology induced by $\cE_\cS$.}  It
is easily seen that this assignment is part of a forgetful functor $U:
\perv \to \Top$ which is right adjoint to the embedding $\Top
\hookrightarrow \perv$. Thus, \Top{} a full coreflective subcategory
of \perv{}.

Next, we will characterize the \emph{extremal monomorphisms} of Pervin
spaces. For that, we first need to describe the epimorphisms.

\begin{lemma}\label{episurj} Let $e: (X, \cS) \to (Y, \cT)$ be a
  morphism of Pervin spaces. Then, $e$ is an epimorphism if and only
  if the set map $e: X \to Y$ is surjective.
\end{lemma}
\begin{proof}
  The argument to show that every surjection is an epimorphism is
  analogous to the set-theoretical one. For the converse, let
  $e:(X,\ca{S})\ra (Y,\ca{T})$ be an epimorphism. We consider the
  two-point Pervin space $(Z, \cU) := (\{0,1\},\{\emptyset,
  \{0,1\}\})$, and we let $f_1, f_2: Y \to Z$ be defined by $f_1(y)=1$
  if and only if $y\in e[X]$, and by $f_2(y)=1$ for all $y\in
  Y$. Then, $f_1$ and $f_2$ induce morphisms of Pervin spaces
  $(Y,\ca{T}) \to (Z, \cU)$ satisfying $f_1\circ e = f_2\circ
  e$. Since $e$ is an epimorphism, we must have $f_1 = f_2$. But this
  implies $e[X] = Y$, that is, $e$ is a surjection.
\end{proof}

\begin{proposition}\label{p:20}
  A map $m:(X,\ca{S})\ra (Y,\ca{T})$ of Pervin spaces is an extremal
  monomorphism if and only if the set map $m: X \to Y$ is injective
  and every $S \in \cS$ is of the form $m^{-1}(T)$ for some $T \in
  \cT$.
\end{proposition}
\begin{proof}
  Given an extremal monomorphism $m:(X,\ca{S})\ra (Y,\ca{T})$, we let
  $Z := m[X]$, and $\cU$ be the bounded sublattice of $\cP(Z)$
  consisting of the subsets of the form $T \cap m[X]$ for some $T \in
  \cT$, so that we have a Pervin space $(Z, \cU)$. Then, the direct
  image factorization of $m: X \to Y$ induces morphisms of Pervin
  spaces $e: (X, \cS) \twoheadrightarrow (Z, \cU)$ and $g: (Z, \cU)
  \hookrightarrow (Y, \cT)$ satisfying $m = g\circ e$. By
  Lemma~\ref{episurj}, $e$ is an epimorphism and thus, an
  isomorphism. In particular, the underlying set-theoretical map of
  $e$ is a bijection, and since $m = g\circ e$, it follows that $m$ is
  injective. It remains to show that $m^{-1}(\cT) = \cS$. For that, we
  let $n$ be the inverse of $e$ and we pick $S \in \cS$. Since $n$ is
  a morphism of Pervin spaces, we have that $n^{-1}[S] = e[S] = m[S]$
  belongs to $\cU$. Therefore, there exists $T \in \cT$ such that
  $m[S] = T \cap m[X]$. But since $m$ is injective, this implies $S =
  m^{-1}(T)$ as required.

  Conversely, suppose that $m:(X,\ca{S})\hookrightarrow (Y,\ca{T})$ is
  injective and satisfies $\cS = m^{-1}(\cT)$. Consider a
  factorization $m=g\circ e$, where $e:(Z,\ca{U})\to (Y,\ca{T})$ is an
  epimorphism. By Lemma~\ref{episurj}, $e$ is a surjection, and it is
  also injective by injectivity of $m$. We let $n$ be the
  set-theoretical inverse of $e$ and we show that $n$ is a morphism of
  Pervin spaces. Given $S \in \cS$, there is, by hypothesis, an
  element $T \in \cT$ such that $S = m^{-1}(T)$. Then, since $g$ is a
  morphism of Pervin spaces, we have that
  \[n^{-1}(S) = (mn)^{-1}(T) = g^{-1}(T)\]
  belongs to $\cU$. Thus, $n$ is an isomorphism in $\perv$.
\end{proof}

We will say that $(Y, \cT)$ is a \emph{subspace} of $(X, \cS)$ if
there exists an extremal monomorphism $(Y, \cT) \hookrightarrow (X,
\cS)$.
We remark that subspaces of $(X, \cS)$ are, up to isomorphism, in a
bijection with the subsets of~$X$. We also note that every extremal
monomorphism $m:(X,\ca{S})\hookrightarrow (Y,\ca{T})$ is regular. To
see this, we consider the two-point Pervin space $(Z, \cU) :=
(\{0,1\}, \{\emptyset, \{0, 1\}\})$. Then, $m$ is the equalizer of the
morphisms of Pervin spaces $f_1, f_2: (Y, \cT) \to (Z, \cU)$ defined
by $f_1(y) = 1$ for every $y \in Y$ and by $f_2(y) = 1$ if and only if
$y \in m[X]$.

Since every epimorphism which is also an extremal monomorphisms is an
isomorphism (cf.~\cite[Proposition~4.3.7]{Borceux94-vol1}), we obtain
the following:

\begin{corollary}\label{perviniso}
  Isomorphisms in $\bd{Pervin}$ are the bijections $f:(X,\ca{S})\ra
  (Y,\ca{T})$ such that $f[S]$ belongs to $\cT$, for every $S \in
  \cS$.
\end{corollary}
\begin{proof}
  By Lemma~\ref{episurj} and Proposition~\ref{p:20}, we know that a
  morphism of Pervin spaces $f:(X,\ca{S})\ra (Y,\ca{T})$ is an
  isomorphism if and only if it is bijective and satisfies $\cS =
  f^{-1}(\cT)$. Thus, the claim follows from having that, if $f$ is a
  bijection, then $f[\cS] \subseteq \cT \iff \cS \subseteq
  f^{-1}[\cT]$.
\end{proof}

We finally consider the full subcategory $\sperv$ of \perv{} whose
objects are those Pervin spaces $(X, \cB)$ such that $\cB$ a Boolean
algebra. In the setting of Pervin quasi-uniform spaces, this is a
relevant subcategory because it corresponds to the uniform spaces. In
our paper, it will play a role when the pointfree version of this fact
is discussed (cf. Section~\ref{sec:4}, namely Theorem~\ref{t:5}).

An important property of $\sperv{}$ is that of being coreflective in
\perv{} as we will see now (see Proposition~\ref{p:8} for the
pointfree version of this result). We define a functor $\psym:\perv\ra
\sperv$ as follows. For an object $(X,\cS)$, we let $\psym(X,\cS)$ be
the Pervin space $(X,\overline{\cS})$, where $\overline{\ca{S}}$ is
the Boolean subalgebra of the powerset $\ca{P}(X)$ generated by the
elements of~$\ca{S}$.  On morphisms, we simply map a function to
itself. Notice that, since a map of Pervin spaces $f:(X,\ca{S})\ra
(Y,\ca{T})$ is such that the preimage of a complement of an element in
$\ca{T}$ is the complement of an element in $\ca{S}$, this assignment
is well-defined on morphisms. Clearly, $\psym$ is a functor. We show
now that $\psym$ is right adjoint to the embedding $\sperv
\hookrightarrow \perv$.

\begin{proposition}\label{p:27}
  The category of symmetric Pervin spaces is a full coreflective
  subcategory of that of Pervin spaces. More precisely, the functor
  $\psym$ is right adjoint to the embedding $\sperv \hookrightarrow
  \perv$.
\end{proposition}
\begin{proof}
  We first notice that, for every Pervin space $(X, \cS)$, the
  identity function on~$X$ induces a morphism of Pervin spaces
  $\id_X:(X, \overline{\cS})\ra (X,\ca{S})$. Thus, we only need to
  show that, for every morphism of Pervin spaces $f:(Y,\ca{B})\ra
  (X,\ca{S})$, with $\cB$ a Boolean algebra, there is a unique Pervin
  map $\widetilde{f}:(Y,\cB)\ra (X,\overline{\cS})$ satisfying $f =
  \id_X \circ \widetilde f$.  Of course, this holds if and only if $f$
  also induces a morphism of Pervin spaces $f:(Y,\ca{B})\ra
  (X,\overline{\ca{S}})$.  This is indeed the case because
  $\overline{\cS}$ is generated, as a lattice, by $\cS$ together with
  the complements in $\cP(X)$ of the elements of $\cS$ and, since
  $\cB$ is a Boolean subalgebra of $\cP(Y)$, it follows that
  $f^{-1}(S^c) = (f^{-1}(S))^c$ belongs to $\cB$ for every $S \in
  \cS$.
\end{proof}

Recall that the \emph{Skula topology} on a given topological space
$(X, \tau)$ is the topology generated by $\tau$ together with the
complements of its elements. Therefore, if $(X, \cS)$ is a Pervin
space, then the topology $\Omega_{\overline{\cS}}(X)$ on $X$ defined
by its symmetrization is precisely the Skula topology on the
topological space $(X, \Omega_\cS(X))$ defined by $(X, \cS)$.  For
that reason, we will say that $\Omega_{\overline{\cS}}(X)$ is the
\emph{Skula topology} on $X$ induced by $(X, \cS)$.
%%%%%%%%%%%%%%%%%%%%%%%%%%%%%%%%%%%%%%%%%%%%%%%%%%%%%%%%%%%%%%%%%%%%%%
\section{Frith frames as a pointfree version of Pervin
  spaces}\label{sec:1}

\subsection{The category of Frith frames and the $\perv$ - $\ffrm$
  dual adjunction}\label{FrFrm}

A \emph{Frith frame} is a pair $(L,S)$ where $L$ is a frame and $S\se
L$ is a bounded sublattice such
that all elements in $L$ are joins of elements in $S$. A morphism
$h:(L,S)\ra (M,T)$ of Frith frames is a frame homomorphism $h:L\ra M$
such that whenever $s\in S$ we have $h(s)\in T$. We denote the
category of Frith frames and corresponding morphisms by~\ffrm{}. By
identifying a frame~$L$ with the Frith frame~$(L, L)$, we may see
\Frm{} as a full reflective subcategory of \ffrm{}. Indeed, we have
the following:

\begin{proposition}\label{p:14}
  The category of frames is a full reflective subcategory of that of
  Frith frames. More precisely, the forgetful functor $\ffrm \to \Frm$
  is left adjoint to the embedding $\Frm \hookrightarrow \ffrm$.
\end{proposition}
\begin{proof}
  Let $U: \ffrm \to \Frm$ denote the forgetful functor and $F: \Frm
  \hookrightarrow \ffrm$ the embedding identifying a frame $L$ with
  the Frith frame $(L, L)$.  We only need to observe that for every
  Frith frame $(L, S)$ and every frame~$M$, every frame homomorphism
  $L \to M $ induces a morphism of Frith frames $(L, S) \to (M, M)$,
  and thus, we have a natural isomorphism $\Frm(U-, -) \cong \ffrm(-,
  F-)$.
\end{proof}

Next, we will see that Frith frames may indeed be considered the
pointfree analogues of Pervin spaces, by showing that the classical
adjunction $\Omega: \Top \adj \Frm^{\rm op}: \pt$ extends to an adjunction
$\Omega: \perv \adj \ffrm^{\rm op}: \pt$ between the categories of Pervin
frames and of Frith frames, so that the following diagram commutes:
\begin{center}
  \begin{tikzpicture} \node (A) {$\Top$}; \node[right of = A, xshift =
    20mm] (B) {$\perv$}; \node[below of = A] (C) {$\Frm^{\rm op}$};
    \node[below of = B] (D) {$\ffrm^{\rm op}$};
    \draw[right hook->] (A) to (B); \draw ([xshift=-2mm]B.270) to
    node[ArrowNode,left] {$\Omega$} ([xshift=-2mm]D.90); \draw
    ([xshift=-2mm]A.270) to node[ArrowNode,left] {$\Omega$}
    ([xshift=-2mm]C.90); \draw[right hook->] (C) to (D); \draw (D) to
    node[ArrowNode,right] {$\pt$} (B); \draw (C) to
    node[ArrowNode,right] {$\pt$} (A);
  \end{tikzpicture}
\end{center}
Let us define the open set functor~$\Omega: \perv \to \ffrm$. Given a
Pervin space $(X,\mathcal{S})$ we set $\Om(X,\ca{S}) :=
(\Om_\cS(X),\ca{S})$, where $\Om_\cS(X)$ denotes the topology on~$X$
generated by~$\cS$ (recall Section~\ref{sec:pervin}). If $f: (X, \cS)
\to (Y, \cT)$ is a morphism of Pervin spaces then taking preimages
under~$f$ defines a morphism of Frith frames $\Omega(f):= f^{-1}:
(\Omega_\cT(Y), \cT) \to (\Omega_\cS(X), \cS)$. It is easily seen that
this assignment yields a functor $\Om:\perv\ra \ffrm^{\rm op}$.
In turn, the spectrum functor~$\pt: \ffrm^{\rm op} \to \perv$ is defined
on objects by $\pt(L, S) := (\pt(L),\widehat S)$ for every Frith frame
$(L, S)$, where $\widehat S:= \{\widehat s \mid s \in S\}$. Finally,
if $h:(L,S)\ra (M,T)$ is a morphism of Frith frames, then $\pt(h) :=
(-\circ h)$ is given by precomposition with~$h$. The following lemma
shows that $\pt(h)$ defines a morphism between the Pervin spaces
$(\pt(M), \widehat T)$ and $(\pt(L), \widehat S)$.

\begin{lemma} Let $h:(L, S) \to (M, T)$ be a morphism of Frith
  frames. Then, for every $s \in S$, the equality
  $(\pt(h))^{-1}(\widehat s) = \widehat {h(s)}$ holds. In particular,
  $\pt(h)$ induces a morphism of Pervin spaces $\pt(h):(\pt(M), \widehat T)\to(\pt(L), \widehat S)$.
\end{lemma}
\begin{proof} Let $s \in S$ and $p: M \to \two$ be a point. The claim
is a consequence of the following computation:
  \[ p \in (\pt(h))^{-1}(\widehat s) \iff \pt(h)(p) \in \widehat s
\iff p (h (s)) = 1 \iff p \in \widehat{h(s)}. \popQED \] 
\end{proof}

We may now prove that the functors just defined form an adjunction.

\begin{proposition}\label{p:24} There is an adjunction $\Om:\perv\lra \ffrm^{\rm
    op}:\pt$, which extends the classical adjunction $\Om:\Top\lra
  \Frm^{\rm op}:\pt$.
\end{proposition}
\begin{proof}
Let $(L, S)$ be a Frith frame and $(X, \cS)$ be a Pervin
    space. Let also $\varphi: L \to \Omega \circ \pt(L)$ and $\psi: X
    \to \pt\circ \Omega_\cS(X)$ be, respectively, the components at
    $L$ and at $(X, \Omega_\cS(X))$ of the unit and counit of the
    classical adjunction $\Omega:\Top \adj \Frm^{\rm op} : \pt$. It
    suffices to show that $\varphi$ and $\psi$ induce, respectively,
    morphisms $(L, S) \to (\Omega_{\widehat{S}}(\pt(L)), \widehat S)$
    and $ (X, \cS) \to (\pt(\Omega_\cS(X)), \widehat \cS)$ in the
    suitable categories.

    First notice that, since $S$ join-generates $L$, we have
    $\Omega_{\widehat{S}}(\pt(L)) = \Omega(\pt(S))$. Thus, it follows
    from its definition that $\varphi$ induces a morphism $(L, S) \to
    (\Omega_{\widehat{S}}(\pt(L)), \widehat S)$ of Frith frames. To
    see that $\psi$ induces a morphism $(X, \cS) \to
    (\pt(\Omega_\cS(X)), \widehat \cS)$ of Pervin spaces, it suffices
    to observe that $\psi^{-1}(\widehat S) =S$, for every $S \in
    \cS$.
\end{proof}

% \begin{proof} By definition of $\Om:\perv\lra \ffrm^{\rm op}:\pt$, we
%   only need to show that the {\bf component $\varphi_L$ of the unit at
%     a frame~$L$} and {\bf the component $\psi_X$ of the counit at a
%     space $X$} of the adjunction $\Omega:\Top \adj \Frm^{\rm op} :
%   \pt$ define, respectively, a morphism $\varphi_L: (L, S) \to
%   (\Omega_{\widehat S}(\pt(L)), \widehat S)$ of Frith frames and a
%   morphism $\psi_X: (X, \cS) \to (\pt(\Omega_\cS(X)), \widehat \cS)$
%   of Pervin spaces.

%   By definition of $\varphi_L$, we have $\varphi_L(s) = \widehat s \in
%   \widehat S$, for every $s \in S$. So $\varphi_L$ does define a
%   morphism of Frith frames. It remains to show that
%   $\psi_X^{-1}(\widehat S) \in \cS$, for every $S \in \cS$. That is
%   indeed the case because, for every $x \in X$, we have
%   \[x \in \psi_X^{-1}(\widehat S) \iff \psi_X(x) \in \widehat S \iff
%     \psi_X(x)(S) = 1 \iff x \in S.\popQED\]
% \end{proof}

We finish this section by remarking that a Pervin space $(X,
  \cS)$ is a fixpoint of this adjunction if and only if the topological space
  $(X, \Omega_\cS(X))$ it defines is sober; we call such Pervin spaces \emph{sober}. A Frith frame $(L, S)$ is a fixpoint of the adjunction above if and only if its underlying
  frame~$L$ is spatial; we call such Frith frames \emph{spatial}. Indeed, this can be easily seen from the description of
  the unit and counit of the adjunction $\Om:\perv\lra \ffrm^{\rm
    op}:\pt$, together with the characterization of the isomorphisms
  in the categories of Pervin spaces (cf. Corollary~\ref{perviniso})
  and of Frith frames that we will shortly provide
  (cf. Corollary~\ref{p:7}).

%%%%%%%%%%%%%%%%%%%%%%%%%%%%%%%%%%%%%%%%%%%%%%%%%%%%%%%%%%%%%%%%%%%%%%
\subsection{Compact, coherent, and zero-dimensional Frith
  frames}\label{sec:coh}

In this section we discuss the appropriate notions of compactness,
coherence, and zero-dimensionality for Frith frames. Let $L$ be a
frame and $S \subseteq L$. We say that an element $a \in L$ is
\emph{$S$-compact} if whenever $a \leq \bigvee_{i \in I} s_{i}$ for
some $\{s_i\}_{i \in I}\subseteq S$, there exists $F \subseteq I$
finite so that $a \leq \bigvee_{i \in F}s_i$, and we say that $L$ is
\emph{$S$-compact} if its top element is $S$-compact.  Clearly, every
compact element of $L$ is also $S$-compact. If we further assume that
$S$ is join-dense in~$L$ (which is the case when $(L, S)$ is a Frith
frame), then we also have the converse:

\begin{lemma}\label{l:7}
  Let $L$ be a frame, $S \subseteq L$ be a join-dense subset, and $a
  \in L$. Then, $a$ is $S$-compact if and only if $a$ is compact. In
  particular, $L$ is $S$-compact if and only if it is compact.
\end{lemma}
\begin{proof}
  Let $a$ be an $S$-compact element and suppose that $a \leq
  \bigvee_{i \in I} a_i$ for some $i \in I$. Since $S$ is join-dense
  in~$L$, we may write each $a_i$ as a join of elements in $S$. Thus,
  since $a$ is $S$-compact, there exists a finite subset $F \subseteq
  I$ satisfying $a \leq \bigvee_{i \in F} a_i$.
\end{proof}
We will say that a Frith frame $(L, S)$ is \emph{compact} if its frame
component $L$ is compact, and we say that $(L, S)$ is \emph{coherent}
if $S$ consists of compact elements of~$L$. We call $\bd{CohFrith}$
the full subcategory of $\ffrm$ determined by the coherent Frith
frames. Since $S$ is, by definition of Frith frame, a bounded
sublattice of $L$, we have that every coherent Frith frame is
compact. Also, every coherent frame $L$ gives rise to a coherent Frith
frame $(L, K(L))$. We now show that every coherent Frith frame is of
this form.
\begin{lemma}\label{l:8}
  Let $(L, S)$ be a Frith frame. Then, the following are equivalent:
  \begin{enumerate}
  \item\label{item:15} $S$ consists of compact elements;
  \item\label{item:16} $S$ is the set of all compact elements of~$L$.
  \end{enumerate}
  In particular, $(L, S)$ is a coherent Frith frame if and only if $L$
  is coherent and $S = K(L)$. 
\end{lemma}
\begin{proof}
  The implication \ref{item:16} $\implies$ \ref{item:15} is
  trivial. Conversely, suppose that $S \subseteq K(L)$, and let $a \in
  K(L)$. Since $(L, S)$ is a Frith frame, we have that $S$ is
  join-dense in $L$, and thus we may write $a = \bigvee_{i \in I}s_i$
  for some subset $\{s_i\}_{i \in I} \subseteq S$. But compactness
  of~$a$ yields the existence of a finite subset $F \subseteq I$
  satisfying $a = \bigvee_{i \in F}s_i$. Since $S$ is closed under
  finite joins, it follows that $a$ belongs to $S$.
\end{proof}
A consequence of Lemma~\ref{l:8} is that, if $(L, S)$ and $(M, T)$ are
coherent Frith frames, then a frame homomorphism $h: L \to M$ induces
a morphism between the corresponding Frith frames if and only if it is
coherent. Thus, the categories $\cohfrm$ of coherent frames and
$\cohfrith$ of coherent Frith frames are isomorphic.
In particular, since $\dlat$ and $\cohfrm$ are equivalent categories,
we also have and equivalence $\dlat \cong \cohfrith$.  More generally,
we have the following analogue of Proposition~\ref{p:15}:
\begin{proposition} \label{p:17} There is an adjunction $\idl(-)
  \dashv U$, where $U: \ffrm \to \dlat$ is the forgetful functor and
  $\idl(-): \dlat \to \ffrm$ is defined by
  \[\idl(S) := (\idl(S), S) \qquad \text{ and }\qquad \idl(h: S \to T) :=
    (\idl(h): (\idl(S), S) \to (\idl(T), T)).\]
  Moreover, the corestriction of $\idl(-)$ to \cohfrith{} induces an
  equivalence of categories whose inverse is the suitable restriction
  of~$U$.
\end{proposition}
\begin{proof}
  It is clear that $\idl(-)$ is a well-defined functor. The fact that
  $\idl(-)$ is left adjoint to~$U$ follows from the existence of a
  natural isomorphism $\ffrm(\idl(-), -) \cong \dlat(-, U-)$, which is
  a straightforward consequence of Proposition~\ref{p:15}.

  Now, again by Proposition~\ref{p:15}, the functors $\idl(-):\dlat
  \to \cohfrm$ and $K(-): \cohfrm \to \dlat$ are mutually inverse, up
  to natural isomorphism. Composing these with the isomorphism
  $\cohfrm \equiv \cohfrith$, we obtain the equivalence described in
  the last statement.
\end{proof}

Just like for frames, we also have in this setting that $U: \cohfrm
\to \dlat$ is both a left and a right adjoint of $\idl(-): \dlat \to
\cohfrm$.  Therefore, we have the following version of
Proposition~\ref{p:16}:
\begin{proposition}
  The category $\cohfrith$ is a full coreflective subcategory of
  $\ffrm$. The coreflector is the functor $\idl(-)\circ U:\ffrm \to
  \cohfrith$.
\end{proposition}

On the other hand, unlike what happens for the frame ideal completion
$\idl(-): \Frm \to \cohfrm$, the functor $\idl(-) \circ U: \ffrm \to
\cohfrm$ is idempotent. Indeed, that may be seen as a consequence of
having a full embedding $\cohfrm \hookrightarrow \Frm$:
\begin{proposition}[{\cite[Proposition 3.4.1]{Borceux94-vol1}}]
  Let $F:\ca{D}\to \ca{C}$ and $G: \cC\to \cD$ be two functors, and
  suppose that $F$ is the left adjoint of $G$. Then, $F$ is full and
  faithful if and only if the unit of $F \dashv G$ is a natural
  isomorphism.
\end{proposition}

We finally define what is a zero-dimensional Frith frame. First
observe that, if $S \subseteq B(L)$ and $(L, S)$ is a Frith frame,
then the frame $L$ is zero-dimensional. Thus, we will say that $(L,
S)$ is \emph{zero-dimensional} provided $S$ consists of complemented
elements, and we have that $L$ is a zero-dimensional frame if and only
if $(L, B(L))$ is a zero-dimensional Frith frame.
Let us analyze the relationship between compactness, coherence, and
zero-dimensionality of Frith frames. If $(L, S)$ is compact and
zero-dimensional then, $L$ is compact and zero-dimensional, too, and,
by Proposition~\ref{p:12}, we have $B(L) = K(L)$. Therefore, $S$
consists of compact elements and thus, $(L, S)$ is coherent. Hence, we
have the following analogue of Proposition~\ref{p:12}:

\begin{lemma}\label{l:13}
  Let $(L, S)$ be a Frith frame. If $(L, S)$ is compact and
  zero-dimensional, then it is coherent.
\end{lemma}
%%%%%%%%%%%%%%%%%%%%%%%%%%%%%%%%%%%%%%%%%%%%%%%%%%%%%%%%%%%%%%%%%%%%%%
\subsection{Limits and colimits}\label{sec:lim-colim}

In this section we show that the category of Frith frames has all
(co)products and (co)equalizers, and provide their description. In
particular, it follows that {\ffrm} is a complete and cocomplete
category.

Let us start with products and equalizers. By
Proposition~\ref{p:17}, the forgetful functor $\ffrm \to \dlat$ is
a right adjoint. Since right adjoints preserve limits, we
automatically know how to compute the lattice component of every
existing limit in {\ffrm}. With this in mind, we may easily prove
  the following characterizations:

% Let $\{(L_i, S_i)\}_{i \in I}$ be a family of Frith frames. As just
% mentioned, if the product $\prod_{i \in I}(L_i, S_i)$ exists in the
% category of Frith frames, then it is of the form $(L,\ \prod_{i \in I}
% S_i)$, for some frame~$L$. Moreover, each of the product maps
% $\pi_i:(L,\ \prod_{i \in I}S_i) \to (L_i, S_i)$ is such that its
% restriction to $\prod_{i \in I}S_i$ is the $i$-th projection
% to~$S_i$. Note that we may see $\prod_{i \in I}S_i$ as a sublattice of
% $\prod_{i \in I}L_i$. The natural candidate for $L$ is then the
% subframe of $\prod_{i \in I}L_i$ generated by $\prod_{i \in
%   I}S_i$. Since each $(L_i, S_i)$ is a Frith frame, this is the whole
% frame $\prod_{i \in I}L_i$ and $\pi_i$ defined on $\prod_{i \in I}L_i$
% is simply the $i$-th projection to $L_i$. We show that this is indeed
% the product of the family $\{(L_i, S_i)\}_{i \in I}$:
\begin{proposition}\label{pr}
  Let $\{(L_i, S_i)\}_{i \in I}$ be a family of Frith frames. Then,
  the product $\prod_{i\in I} (L_i,S_i)$ in the category of Frith
  frames is $(\prod_{i \in I}L_i, \prod_{i \in I}S_i)$, and the
  product map $\pi_i: (\prod_{i \in I}L_i, \prod_{i \in I}S_i) \to
  (L_i, S_i)$ is the $i$-th projection.
\end{proposition}
% \begin{proof}
%   Let $\{h_i:(M,T)\ra (L_i,S_i)\}_{i \in I}$ be a collection of
%   morphisms of Frith frames, and $h:M \ra \prod_{i \in I} L_i$ be the
%   unique frame homomorphism such that $\pi_i \circ h = h_i$, for every
%   $i \in I$, and whose existence is guaranteed by the universal
%   property of products in $\Frm$. We only need to check that this is a
%   map of Frith frames. Indeed, for every $t\in T$ and $i \in I$, we
%   have $h_i(t)\in S_i$ and so, the tuple $h(t) = (h_i(t))_{i \in I}$
%   belongs to $\prod_{i\in I} S_i$.
% \end{proof}

% We now let $h_1, h_2: (L, S) \to (M, T)$ be a pair of morphisms of
% Frith frames. Following a similar reasoning, we know that, if the
% equalizer of $h_1$ and $h_2$ exists, then it is of the form $e: (K,
% S_{h_1 = h_2}) \to (L, S)$, where $S_{h_1 = h_2} := \{s \in S \mid
% h_1(s) = h_2(s)\}$ is the equalizer of $h_1, h_2$ in {\dlat}, and
% $e$ restricted to $S_{h_1 = h_2}$ is the inclusion map. Thus, we
% take for $K$ the subframe of $L$ generated by $S_{h_1 = h_2}$, for
% $e$ the inclusion map, and we show that the morphism of Frith frames
% $e: (K, S_{h_1 = h_2}) \to (L, S)$ satisfies the universal property
% of equalizers in {\ffrm}.

\begin{proposition}\label{eq}
  Given two morphisms $h_1,h_2:(L,S)\rightrightarrows (M,T)$ in the
  category of Frith frames, let $S_{h_1 = h_2}$ denote the sublattice
  $\{s \in S \mid h_1(s) = h_2(s)\}$ of $S$, and $K$ be the subframe
  of $L$ generated by $S_{h_1 = h_2}$. Then, the equalizer of $h_1$
  and $h_2$ is the subframe inclusion $e:(K,S_{h_1=h_2})\inclu (L,S)$.
\end{proposition}
% \begin{proof}
%   Since $S_{h_1 = h_2}$ is, by definition, join-dense in $K$, we
%   have $h_1 \circ e = h_2 \circ e$. If $e':(L',S')\ra (L,S)$ is a
%   morphism of Frith frames satisfying $h_1\circ e'=h_2\circ e'$,
%   then we must have $e'[S']\se S_{h_1 = h_2}$, and since $L'$ is
%   generated, as a frame, by $S'$, we also have $e'[L]\se
%   K$. Therefore, the co-restriction of $e'$ to $K$ is the unique
%   morphism $u:(L',S')\ra (K,S_{h_1 = h_2})$ such that $e'=e\circ u$.
% \end{proof}

Colimits in $\ffrm$ are also computed as expected. Note that, by
  Proposition~\ref{p:14}, the forgetful functor $\ffrm \to \Frm$ is a
  left adjoint. Thus, it preserves existing colimits.

% In order to compute coproducts and coequalizers in $\ffrm$, we follow
% a similar strategy. By Proposition~\ref{p:14}, the forgetful functor
% $\ffrm \to \Frm$ is a left adjoint. Thus, it preserves existing
% colimits.

% Let $\{(L_i, S_i)\}_{i \in I}$ be a family of Frith frames. If its
% coproduct exists, then it has to be a Frith frame of the form
% $(\bigoplus_{i \in I}L_i, S)$, for some join-dense sublattice $S$ of
% $\bigoplus_{i \in I}L_i$, and the coproduct maps must be given by the
% coproduct injections $\iota_i: (L_i, S_i) \to (\bigoplus_{i \in I}L_i,
% S)$ in $\Frm$. In particular, $S$ must contain the sublattice
% $\iota_i[S_i]$, for every $i \in I$. We have the following:

\begin{proposition}\label{copr}
  Let $\{(L_i, S_i)\}_{i \in I}$ be a family of Frith frames, and
  $\{\iota_i: L_i \hookrightarrow \bigoplus_{i \in I}L_i\}_{i \in I}$
  be the coproduct injections in {\Frm}.  The coproduct $\bigoplus_{i
    \in I} (L_i,S_i)$ of Frith frames in $\bd{Frith}$ is
  $(\bigoplus_{i \in I}L_i, S)$, where $S$ is the sublattice of the
  coproduct $\bigoplus_{i\in I}L_i$ generated by $\bigcup_i
  \iota_i[S_i]$. Moreover, the $i$-th coproduct map is the morphism of
  Frith frames $\iota_i:(L_i, S_i) \to (\bigoplus_{i \in I}L_i, S)$
  defined by $\iota_i$.
\end{proposition}
% \begin{proof}
%   We first note that $(\bigoplus_{i \in I}L_i,S)$ is a Frith
%   frame. Indeed, this is because $\bigoplus_{i \in I} L_i$ is
%   generated, as a frame, by $\bigcup_{i \in I} \iota_i[L_i]$, and each
%   $\iota_i[L_i]$ is generated, as a frame, by $\iota_i[S_i]$. Suppose
%   that we have a collection of maps of Frith frames $h_i:(L_i,S_i)\ra
%   (M,T)$. We let $h:\bigoplus_{i \in I}L_i \to M$ be the unique frame
%   homomorphism such that $h \circ \iota_i = h_i$, for every $i \in I$,
%   as given by the universal property of coproducts in {\Frm}. It
%   suffices to show that $h$ defines a morphism of Frith frames $h:
%   (\bigoplus_{i \in I}L_i, S) \to (M, T)$. But this is a
%   straightforward consequence of having $h\circ \iota_i[S_i] =
%   h_i[S_i]\se T$, for every $i \in I$.
% \end{proof}

% Finally, it remains to compute the coequalizers in {\ffrm}. Given a
% pair of parallel morphisms $h_1, h_2: (L, S) \to (M, T)$, we let $q: M
% \twoheadrightarrow K$ be the coequalizer in {\Frm} of the frame
% homomorphisms $h_1$ and $h_2$. Explicitly, $K$ is the quotient of $M$
% by the congruence generated by the subframe $\{(h_1(a), h_2(a)) \mid a
% \in L\}$ of $M \times M$. Once again, we know that, if the coequalizer
% of $h_1, h_2$ in {\ffrm} exists, then it is the morphism $q: (M, T)
% \twoheadrightarrow (K, R)$ induced by~$q$, for a suitable
% sublattice~$R$ of~$K$. It is easy to see that, if we take $R = q[T]$,
% then $(K, R)$ is a Frith frame and $q$ is a morphism of Frith frames
% that satisfies the universal property of coequalizers in
% $\ffrm$. Therefore, we have the following:
\begin{proposition}\label{coeq}
  Let $h_1,h_2:(L,S)\to (M,T)$ be two morphisms of Frith frames. Then,
  their coequalizer is $q: (M, T) \to (K, R)$, where $q: M
  \twoheadrightarrow K$ is the coequalizer of the frame homomorphisms
  $h_1, h_2$ in $\Frm$ and $R = q[T]$.
\end{proposition}

As a consequence of Propositions~\ref{pr}, \ref{eq}, \ref{copr}, and
\ref{coeq}, we have:
\begin{corollary}
  The category $\bd{Frith}$ is complete and cocomplete.
\end{corollary}

%%%%%%%%%%%%%%%%%%%%%%%%%%%%%%%%%%%%%%%%%%%%%%%%%%%%%%%%%%%%%%%%%%%%%%
\subsection{Special morphisms}\label{sec:spec-mor}

This section is devoted to the study of some special morphisms in the
category of Frith frames.  More precisely, we will start by
characterizing the monomorphisms, the extremal epimorphisms, and the
isomorphisms.  In the setting of frames, extremal epimorphisms are
relevant because they are the pointfree notion of \emph{subspace
  embedding}. We will show a Pervin-Frith analogue of this. We will
also see that, unlike what happens for frames, not every extremal
epimorphism is regular.

% The proofs of the first result is analogous to the usual proof for
% posets.
Using standard arguments, we may show the following:
\begin{lemma}\label{anothercoref}
  A morphism $m:(L, S)\ra (M, T)$ of Frith frames is a monomorphism if
  and only if the map $m: L \to M$ is injective.
\end{lemma}

\begin{proposition}\label{p:21}
  A morphism $e:(L, S)\ra (M, T)$ of Frith frames is an extremal
  epimorphism if and only if it satisfies $e[S] = T$. In particular,
  all extremal epimorphisms are surjective.
\end{proposition}
% \begin{proof}
%   Let $e:(L, S)\ra (M, T)$ be a morphism satisfying $e[S] = T$ and $m:
%   (K, R) \hookrightarrow (M, T)$ be a monomorphism such that $e=m\circ
%   g$ for some other morphism $g: (L, S) \to (K, R)$. In particular, we
%   have
%   \[T = e[S] = m\circ g[S] \subseteq m[R].\]
%   Since $M$ is generated by $T$,
%   it follows that $m: K \to M$ is surjective and thus, a frame
%   isomorphism. Let $f: M \to K$ be its inverse. We claim that $f$
%   induces a morphism of Frith frames $f: (M, T) \to (K, R)$. Indeed,
%   that is a consequence of having
%   \[f[T] = f\circ e[S] = g[S] \subseteq R.\]
%   Clearly, as morphisms of Frith frames, $m$ and $f$ are also mutually
%   inverse. Therefore, $m$ is an isomorphism. For the converse, suppose
%   that $e: (L,S) \to (M, T)$ is an extremal epimorphism. Then, $(e[L],
%   e[S])$ is a Frith frame to which $e$ corestricts and, by
%   Lemma~\ref{anothercoref}, we have a monomorphism $m:
%   (e[L],e[S])\inclu (M, T)$. Thus, $m$ has to be an isomorphism, and
%   so the map $e$ satisfies $e[S]=T$. Finally, since $T$ generates $M$,
%   the map $e$ is a frame surjection.
% \end{proof}
We will say that $(M, T)$ is a \emph{quotient} of $(L, S)$ if there is
an extremal epimorphism $e:(L, S)\twoheadrightarrow (M, T)$. Notice
that quotients of $(L, S)$ are, up to isomorphism, in a one-to-one
correspondence with the congruences on~$L$. The characterization of
Proposition~\ref{p:21}, together with that of the extremal
monomorphisms in $\perv$ made in Proposition~\ref{p:20}, allows us to
conclude that the extremal epimorphisms in \ffrm{} are indeed the
pointfree version of \emph{embeddings of Pervin spaces}. We will say
that a Pervin space is $T_0$ provided so is the topological space it
defines.
\begin{corollary}\label{c:6}
  Let $(X, \cS)$ and $(Y, \cT)$ be Pervin spaces. If
  $m:(X,\ca{S})\hookrightarrow (Y,\ca{T})$ is an extremal monomorphism
  in \perv{} then $\Om(m): (\Omega_\cT(Y), \cT) \twoheadrightarrow
  (\Omega_\cS(X), \cS)$ is an extremal epimorphism in \ffrm{}. The
  converse holds provided $(X, \cS)$ is $T_0$.
\end{corollary}
\begin{proof}
  By Propositions~\ref{p:20} and~\ref{p:21}, the only non-trivial part
  is to show that if $\Omega(m)$ is an extremal epimorphism then $m$
  is injective. Let $x_1, x_2 \in X$ be two distinct points. Since
  $(X, \cS)$ is $T_0$, there exists some $U \in \Omega_\cS(X)$ such
  that $x_1 \in U$ and $x_2 \notin U$. By Proposition~\ref{p:21},
  there exists some $V \in \Omega_\cT(Y)$ such that $\Omega(m)(V) =
  m^{-1}(V) = U$. But then, $m(x_1) \in V$ and $m(x_2) \notin V$ and
  thus, $m(x_1) \neq m(x_2)$ as we intended to show.
\end{proof}

Since every morphism which is both a monomorphism and an extremal
epimorphism is an isomorphism
(cf. \cite[Proposition~4.3.7]{Borceux94-vol1}), we may also conclude
the following:
\begin{corollary}\label{p:7}
  Let $h: (L, S) \to (M, T)$ be a morphism of Frith frames. Then, $h$
  is an isomorphism if and only if $h$ is one-to-one and satisfies $h[S]
  = T$.
\end{corollary}

We finish this section by exploring the relationship between extremal
and regular epimorphisms. Recall that every regular epimorphism is
extremal (cf. \cite[Proposition~4.3.3]{Borceux94-vol1}). In order to
state the conditions under which the converse holds, we will need the
following notion of \emph{Frith congruence}:

\begin{definition}
  A \emph{Frith congruence} on a Frith frame $(L,S)$ is a frame
  congruence on $L$ generated by a relation $\rho\se S\times
  S$. 
\end{definition}

\begin{lemma}\label{rho}
  Let $(L,S)$ be a Frith frame. Then, a congruence on $L$ is a Frith
  congruence if and only if it is generated by its restriction to
  $S\times S$.
\end{lemma}
\begin{proof}
  The backwards direction of the implication is trivial. For the
  converse, suppose that $\theta$ is a Frith congruence. Let $\rho\se
  S\times S$ be such that $\overline{\rho}=\theta$. We have that
  $\theta\cap (S\times S)\se \theta$ and since closure operators are
  monotone and idempotent this implies $\overline{\theta\cap (S\times
    S)}\se \theta$. For the reverse inclusion, since $\theta$ extends
  $\rho$ we have $\rho\se \theta\cap (S\times S)$, and since closures
  are monotone we have $\theta=\overline{\rho}\se \overline{\theta\cap
    (S\times S)}$.
\end{proof}

In the example below we show that not every congruence on $L$ is a
Frith congruence.
\begin{example}\label{sec:5}
  Let $L$ be the frame whose underlying poset is the ordinal
  $\omega+2$, and let $S$ be the sublattice $L \backslash \{\omega\}$
  of $L$. Since $\omega = \bigvee\{n \mid n \in \omega\}$ is the join
  of elements in $S$, the pair $(L, S)$ is a Frith frame. Then, the
  open congruence $\del_{\omega}$ is not a Frith congruence. Indeed,
  since $\del_{\omega} = \{(x,y) \in L \times L \mid x \wedge \omega =
  y \wedge \omega\} = \{(x, y) \in L \times L \mid x = y \text{ or }
  x, y \geq \omega\}$ does not identify any two distinct elements
  of~$S$, if it were a Frith congruence and thus generated by its
  restriction to $S \times S$, then it had to be the identity. But
  that is not the case as it contains, for instance, the element
  $(\omega, \omega + 1)$.
\end{example}

In fact, we have the following:

\begin{proposition}\label{p:13}
  Let $(L, S)$ be a Frith frame and $\theta \subseteq L \times L$ be a
  congruence on $L$. Then, $\theta$ is a Frith congruence if and only
  if it belongs to $\cC_SL$.
\end{proposition}
\begin{proof}
  Note that the congruence generated by a relation $\rho \subseteq L
  \times L$ is $\bigvee \{\nabla_a \wedge \Delta_b \mid (a, b) \in
  \rho\}$. Thus, by definition of Frith congruence, it suffices to
  observe that the fact that $\nabla$ preserves arbitrary joins and
  $S$ is join-dense in $L$ implies that $\cC_SL$ is generated by the
  set $\{\nabla_s,\ \Delta_s \mid s \in S\}$.
\end{proof}

We may now characterize those extremal epimorphisms that are regular.

\begin{proposition}\label{regularepi2}
  Let $q: (L, S) \twoheadrightarrow (M, T)$ be an extremal epimorphism
  of Frith frames. Then, $q$ is a regular epimorphism if and only if
  $\ker (q)$ is a Frith congruence.
\end{proposition}
\begin{proof}
  Suppose that $q: (L, S) \twoheadrightarrow (M, T)$ is a regular
  epimorphism, let us say that $q$ is the coequalizer of $h_1, h_2:
  (K, R) \to (L, S)$. It follows from Proposition~\ref{coeq} that
  $\ker(q)$ is the congruence generated by the subframe $\{(h_1(a),
  h_2(a)) \mid a \in K\}$ of $L \times L$.  Since $R$ is join dense in
  $K$, this subframe is generated by $\{(h_1(r), h_2(r)) \mid r \in
  R\}$ and thus, $\ker(q)$ is generated by $\{(h_1(r),h_2(r))\mid r\in
  R\}$.  Since $h_1$ and $h_2$ are morphisms of Frith frames, the set
  $\{(h_1(r), h_2(r)) \mid r \in R\}$ is a relation on~$S$. Thus,
  $\ker(q)$ is a Frith congruence.

  Conversely, suppose that $\ker (q)$ is a Frith congruence, and let
  $K$ be the subframe of $L \times L$ generated by $R:= \ker(q) \cap
  (S \times S)$. Clearly, the pair $(K, R)$ is a Frith frame, and the
  two projection maps $K \to L$ induce morphisms of Frith frames
  $\pi_1,\pi_2:(K, R)\to (L,S)$.  We claim that $q$ is the coequalizer
  of~$\pi_1$ and~$\pi_2$. By Proposition~\ref{coeq} and using the fact
  that $q$ is an extremal epimorphism, it suffices to show that
  $\ker(q) = \overline{\rho}$, where $\rho:=
  \{(\pi_1(x,y),\pi_2(x,y))\mid(x,y)\in K\}$.  We observe that $\rho =
  K$ and, since $R$ generates $K$ as a frame, it follows that $\rho$
  and $R$ generate the same congruence on~$L$. Finally, since
  $\ker(q)$ is a Frith congruence, by Lemma~\ref{rho}, we may then
  conclude that
  \[\ker(q)
    = \overline{\ker(q) \cap (S \times S)} = \overline{R} =
    \overline{\rho},\] as required.
\end{proof}

A \emph{Frith quotient} of $(L, S)$ is then a Frith frame $(M, T)$ for
which there is a regular epimorphism $q: (L, S) \twoheadrightarrow (M,
T)$. By Propositions~\ref{p:13} and~\ref{regularepi2}, there is a
bijection between Frith quotients of $(L,S)$, up to isomorphism, and
the congruences of $\cC_SL$. We finally show that an analogue of
Corollary~\ref{c:6} does not hold with respect to regular morphisms,
even if we restrict to sober spaces.

\begin{example}\label{sec:6} Consider the set
  $X := \omega+1$, equipped with the lattice~$\cS \subseteq \cP(X)$
  consisting of the downsets of~$X$. Since the
  topology~$\Omega_\cS(X)$ on~$X$ has as open subsets the elements of
  $\cS$ together with $\omega$, we have $\Omega(X, \cS) = (\omega + 2,
  (\omega+2){\setminus}\{\omega\})$. We let $(Y, \cT)$ be the Pervin
  subspace of $(X, \cS)$ defined by the subset $Y:= \omega \subseteq
  X$ and we let $m: (Y, \cT) \hookrightarrow (X, \cS)$ be the
  corresponding subspace embedding. Then,
  \begin{align*}
    \ker(\Omega(m)) & = \{(U_1,U_2) \mid U_1, U_2 \in \Omega_\cS(X), \
                      m^{-1}(U_1) = m^{-1}(U_2)\}
    \\ & = \{(x, y) \mid x,y \in \omega + 2, \ x
         \wedge \omega = y \wedge \omega\} = \Delta_\omega
    \end{align*}
    is the congruence described in Example~\ref{sec:5}, which is not a
    Frith congruence. Therefore, by Proposition~\ref{regularepi2},
    $\Omega(m)$ is not a regular epimorphism. Finally, we argue that
    the topological space $(X, \Omega_\cS(X))$ defined by $(X, \cS)$
    is sober. For a subset $Q \subseteq X$, we let $p_Q: \omega +2 \to
    \two$ be the unique function satisfying $p_Q^{-1}(1) =
    Q$. Clearly, if $p_Q$ is a point of $\omega + 2$, that is, a frame
    homomorphism, then $Q$ must be an upset. Also, it is not hard to
    verify that, among the upsets of $\omega+2$, all but
    $\{\omega,\omega+1\}$ give rise to a point
    ($p_{\{\omega,\omega+1\}}$ is not a point because $\omega
    =\bve_{n\in \omega}n$). Therefore, the points of $\omega+2$ may be
    identified with the elements of $\omega+1$, via the correspondence
    $p_{\up n}\mapsto n$ and $p_{\{\omega+2\}}\mapsto \omega$. Under
    this correspondence the open subsets of $\pt(\omega +2)$ are
    precisely the downsets of $\omega+1$, that is, $(X,
    \Omega_\cS(X))$ is isomorphic to the space $\pt(\omega +2)$ and
    thus, it is sober.
\end{example}

%%%%%%%%%%%%%%%%%%%%%%%%%%%%%%%%%%%%%%%%%%%%%%%%%%%%%%%%%%%%%%%%%%%%%%
\subsection{Frith quotients and the $T_D$ axiom for Pervin
  spaces}\label{sec:td}

Recall that a topological space $X$ is \emph{$T_D$} if for every $x
\in X$ there exists an open subset $U\subseteq X$ containing $x$ such
that $U\backslash\{x\}$ is open.  In the classical topological
setting, every subspace of~$X$ induces, via $\Omega$, a quotient of
$\Omega(X)$, and we have that $X$ is $T_D$ if and only if
\emph{different subspaces} of~$X$ induce \emph{different quotients} of
$\Om(X)$.  For Pervin spaces, this may be translated as follows: by
Corollary~\ref{c:6}, every subspace of a Pervin space $(X, \cS)$
induces, via $\Omega$, a quotient of $(\Omega_\cS(X), \cS)$, and we
have that \emph{different subspaces} of $(X, \cS)$ induce
\emph{different quotients} on $(\Omega_\cS(X), \cS)$ if and only if
the topology $\Omega_\cS(X)$ on $X$ is $T_D$.  In this section we will
state and prove a version of this result where \emph{quotient} is
replaced by \emph{Frith quotient}. Actually, we will show a version of
the following:
\begin{proposition}[{\cite[Chapter I, Section
    1]{picadopultr2011frames}}]
  Let $X$ be a topological space. The following are equivalent:
  \begin{enumerate}
  \item The space $X$ is $T_D$.
  \item Different subspaces of $X$ induce different quotients of
    $\Om(X)$.
  \item For no $x\in X$ does the subspace inclusion $X{\sm}\{x\}\inclu X$ induce an isomorphism $\Omega(X{\sm}\{x\})\cong\Omega(X)$.
  \item The Skula topology on $X$ is discrete.
  \end{enumerate}
\end{proposition}

We have seen in the previous section that, unlike what happens for
frames, the extremal epimorphisms in \ffrm{} do not coincide with the
regular ones (cf. Proposition~\ref{regularepi2}), and
Example~\ref{sec:6} shows that, in general, the functor $\Omega$ does
not map subspaces of a Pervin space $(X, \cS)$ to Frith quotients of
$(\Omega_\cS(X), \cS)$. There is however a natural way to assign to
each subspace of $(X, \cS)$ a Frith quotient on $(\Omega_\cS(X),
\cS)$: given a subset $Y \subseteq X$, we consider the Frith quotient
defined by the Frith congruence $\theta_Y$ generated by
$\ker(\Omega(m)) \cap (\cS \times \cS) = \{(S_1, S_2) \in \cS \times
\cS\mid S_1 \cap Y = S_2 \cap Y\}$, where $m: (Y, \cT) \hookrightarrow
(X, \cS)$ is the subspace embedding determined by~$Y$.  Note that,
since $\ker(\Omega(m))$ is a frame congruence on $\Omega_\cS(X)$, we
have $\theta_Y \cap (\cS \times \cS) = \ker (\Omega(m)) \cap (\cS
\times \cS)$. The following property will be useful later:
\begin{lemma}
  \label{l:12}
  Let $(X, \cS)$ be a Pervin space and $x \in X$. If $\theta_X$ and
  $\theta_{X \backslash \{x\}}$ are distinct, then there are $S_1, S_2
  \in \cS$ such that $\{x\} = S_1 \backslash S_2$.
\end{lemma}
\begin{proof}
  Clearly, $\theta_X$ is the identity relation on
  $\Omega_\cS(X)$. Therefore, since $\theta_{X \backslash \{x\}}$ is,
  by definition, a Frith congruence, there exists some $(S_1, S_2) \in
  \theta_{X \backslash \{x\}} \cap (\cS \times \cS)$ such that $S_1
  \neq S_2$, say $S_1 \nsubseteq S_2$ without loss of generality. By
  definition of $\theta_{X \backslash \{x\}}$, we have $S_1 \backslash
  \{x\} = S_2 \backslash \{x\}$.  Thus, since $S_1 \nsubseteq S_2$, it
  follows that $\{x\} = S_1 \backslash S_2$ as required.
\end{proof}

We say that a Pervin space $(X,\ca{S})$ is \emph{Pervin-$T_D$} if for
every $x\in X$ there is some $S\in \ca{S}$ that contains $x$ and such
that $S \backslash \{x\} \in \cS$.
We first show that in a Pervin-$T_D$ space $(X, \cS)$ any two
different subspaces of $(X, \cS)$ induce different Frith quotients on
$(\Omega_\cS(X), \cS)$.

\begin{lemma}\label{ptd}
  If $(X,\ca{S})$ is a Pervin-$T_D$ space, then different subspaces of
  $(X, \cS)$ induce different Frith quotients on
  $(\Om_\cS(X),\ca{S})$.
\end{lemma}
\begin{proof}
  Let $(X,\ca{S})$ be a Pervin-$T_D$ space, and let $Y, Z \subseteq X$
  be two distinct subsets of~$X$. We need to show that $\theta_Y \neq
  \theta_Z$.  Without loss of generality, we may assume that there
  exists some $x \in Y \backslash Z$. Since $(X, \cS)$ is
  Pervin-$T_D$, we can take $S \in \ca{S}$ such that $x \in S$ and $S
  \backslash \{x\} \in \cS$. Then, as $x$ belongs to $Y$ but not to
  $Z$, we have
  \[S \cap Y \neq (S\backslash \{x\} ) \cap Y\quad \text{and}\quad S
    \cap Z = (S\backslash \{x\} ) \cap Z.\]
  Finally, using that $(S, S\backslash \{x\}) \in \cS \times \cS$ and
  $\theta_Y \cap (\cS \times \cS) = \{(S_1, S_2) \in \cS \times \cS
  \mid S_1 \cap Y = S_2 \cap Y\}$, we may conclude that $(S,
  S\backslash \{x\}) \in \theta_Z \backslash \theta_Y$. Thus,
  $\theta_Y \neq \theta_Z$ as required.
\end{proof}

We have now all the ingredients to show the main result of this
section.
\begin{proposition}
  Let $(X,\cS)$ be a Pervin space. The following are equivalent.
  \begin{enumerate}
  \item\label{item:7} The space $(X,\cS)$ is Pervin-$T_D$.
  \item\label{item:8} Different subspaces of $(X, \cS)$ induce
    different Frith quotients of $(\Om_\cS(X), \cS)$.
  \item\label{item:13} For no $x\in X$ do we have that $\theta_{X{\sm}\{x\}}$ is trivial.
  \item\label{item:9} The Skula topology on $X$ induced by $(X, \cS)$
    is discrete.
  \end{enumerate}
\end{proposition}
\begin{proof}
  The fact that~\ref{item:7} implies~\ref{item:8} is the content of
  Lemma~\ref{ptd}, that~\ref{item:8} implies~\ref{item:13} is trivial,
  and that~\ref{item:13} implies~\ref{item:9} follows easily from
  Lemma~\ref{l:12}. It remains to show that~\ref{item:9}
  implies~\ref{item:7}. Let $x \in X$. Since the Skula topology on $X$
  induced by $(X, \cS)$ is, by definition, generated by the elements
  of $\cS$ and their complements, there exist $S_1, S_2 \in \cS$ such
  that $\{x\} = S_1 \backslash S_2$. In turn, this implies that $S_1$
  is the disjoint union of $\{x\}$ and $S_1 \cap S_2$. Therefore,
  $S_1$ is such that $x \in S_1$ and $S_1 \backslash \{x\} = S_1 \cap
  S_2$ belongs to $\cS$. This shows that $(X, \cS)$ is Pervin-$T_D$ as
  intended.
\end{proof}

%%%%%%%%%%%%%%%%%%%%%%%%%%%%%%%%%%%%%%%%%%%%%%%%%%%%%%%%%%%%%%%%%%%%%%
\section{Frith frames as quasi-uniform frames}\label{sec:2}
In this section we show that the category of Frith frames may be seen
as a coreflective full subcategory of the category of transitive and
totally bounded quasi-uniform frames.

If $K$ is a frame and $r \in K$, then we denote
\[E_r:= (r \oplus 1) \vee (1 \oplus r^*).\]
Note that the set $(r \oplus 1) \cup (1 \oplus r^*)$ is already a
$C$-ideal, thus it equals $E_r$. It is shown in~\cite[Proposition
5.2]{HunsakerPicado2002} that $E_r \circ E_r = E_r$, and
in~\cite[Proposition 5.3]{HunsakerPicado2002} that $E_r$ is an
entourage if and only if $r$ is complemented. Moreover, if $r$ is
complemented, then $\{r, r^*\}$ is a cover of $K$ and, since $(r
\oplus r) \vee (r^* \oplus r^*) \subseteq E_r$, it follows that $E_r$
is a finite entourage.

For a subset $R \subseteq K$ of complemented elements, we denote
$R^*:= \{r^* \mid r \in R\}$, and we let $\cE_R$ be the filter
generated by $\{E_r \mid r \in R\}$. We start by establishing an
important property of the relations $\luni_1$ and $\luni_2$ for the
filter $\cE_R$.
\begin{lemma}\label{l:17}
  Let $K$ be a frame and $R \subseteq K$ be a subset of complemented
  elements. For every $x, a \in K$,
  \begin{enumerate}
  \item\label{item:23} if $x \luni_1 a$, then there exists some $r \in
    \langle R\rangle_{\dlat}$ such that $x \leq r \leq a$,
  \item\label{item:24} if $x \luni_2 a$, then there exists some $r \in
    \langle R^*\rangle_{\dlat}$ such that $x \leq r \leq a$,
  \end{enumerate}
\end{lemma}

\begin{proof} We will only prove~\ref{item:23}, as the proof
  of~\ref{item:24} is similar. By definition, if $x \luni_1 a$, then
  there are some $r_1, \dots, r_n \in R$ such that $(\bigcap_{i = 1}^n
  E_{r_i}) \circ (x \oplus x) \subseteq (a \oplus a)$. We let $[n]$
  denote the set $\{1, \dots, n\}$, and for every $P \subseteq [n]$ we
  set $r_P := \bigwedge_{i \in P} r_i$ and $\overline{r}_P:=
  \bigwedge_{i \notin P}r_i^*$. Since each $r_i$ is complemented, we
  have
  \[1 = \bigwedge_{i = 1}^n (r_i \vee r_i^*) = \bigvee \{r_P \wedge
    \overline{r}_P \mid P \subseteq [n]\}.\]
  Therefore,
  \[x = \bigvee \{r_P \wedge \overline{r}_P\wedge x \mid P \subseteq
    [n]\} \leq \bigvee \{r_P \mid P \subseteq [n], \,
    \overline{r}_P\wedge x \neq 0\}.\]
  Thus, it suffices to show that the element $\bigvee \{r_P \mid P
  \subseteq [n], \, \overline{r}_P\wedge x \neq 0\}$, which belongs to
  the sublattice of~$K$ generated by~$R$, is smaller than or equal
  to~$a$. We let $P \subseteq [n]$ satisfy $\overline{r}_P \wedge x
  \neq 0$. Then, we have $(r_P, \overline{r}_P\wedge x) \in \bigcap_{i
    = 1}^n E_{r_i}$ and $(\overline{r}_P\wedge x, x) \in (x \oplus
  x)$. Since $\overline{r}_P\wedge x \neq 0$, it follows that $(r_P,
  x)$ belongs to $(\bigcap_{i = 1}^n E_{r_i}) \circ (x \oplus x)$
  which, by hypothesis, is contained in $(a \oplus a)$. In particular,
  it follows that $r_P \leq a$, as required.
\end{proof}

The following is a slight variation of~\cite[Theorem
5.5]{HunsakerPicado2002} and the proof in there may be easily
adapted. For the sake of self-containment, we will use the previous
lemma to provide an alternative proof.
\begin{theorem}\label{t:1}
  Let $K$ be a frame and $R \subseteq K$ be a subset of complemented
  elements.  Then, the filter $\cE_R$ generated by the set of
  entourages $\{E_r \mid r \in R\}$ is such that $\cL_1(\cE_R) =
  \langle R\rangle_{\Frm}$ and $\cL_2(\cE_R) = \langle
  R^*\rangle_{\Frm}$. In particular, if $K$ is generated by $R \cup
  R^*$, then $\cE_R$ is a transitive and totally bounded
  quasi-uniformity on~$K$.
\end{theorem}
\begin{proof}
  We first show that $\cL_1(\cE_R) \subseteq \langle
  R\rangle_{\Frm}$. We fix an element $a \in \cL_1(\cE_R)$, and for
  each $x \in K$ satisfying $x \luni_1 a$, we let $r_x \in \langle
  R\rangle_{\dlat}$ be such that $x \leq r_x \leq a$, as given by
  Lemma~\ref{l:17}\ref{item:23}. Then, we have
  \[a = \bigvee \{x\in K \mid x \luni_1 a\} \leq \bigvee \{r_x \mid x
    \in K, \, x \luni_1 a\} \leq a.\]
  Thus, $a = \bigvee \{r_x \mid x \in K, \, x \luni_1 a\} $ is a frame
  combination of elements of $R$. Conversely, since $\cE_R$ is a
  filter and hence, $\cL_1(\cE_R)$ is a frame, it suffices to show
  that $R \subseteq \cL_1(\cE_R)$. But that is a consequence of the
  inclusion $E_r \circ (r \oplus r) \subseteq (r \oplus r)$ which
  holds for every complemented element $r \in K$.
\end{proof}

Let~$L$ be a frame. Then, for each $a \in L$, the congruence
$\nabla_a$ is complemented in~$\cC L$, and $\cC L$ is generated, as a
frame, by these congruences together with their complements. Thus, by
Theorem~\ref{t:1}, the set $\{E_{\nabla_a} \mid a \in L\}$ is a
subbasis for a transitive and totally bounded quasi-uniformity~$\cE_L$
on~$\cC L$. This is called the \emph{Frith quasi-uniformity}, and it
is the pointfree counterpart of the \emph{Pervin quasi-uniformity}, in
the sense that, for every frame~$L$, the frame $\cL_1(\cE_L)$ is
isomorphic to~$L$.
More generally, for every Frith frame $(L, S)$, the set
$\{E_{\nabla_s} \mid s \in S\}$ is a subbasis for a transitive and
totally bounded quasi-uniformity~$\cE_S$ on~$\cC_SL$.
We will now see that the assignment $(L, S) \mapsto (\cC_SL, \cE_S)$
extends to a full embedding $E: \ffrm \hookrightarrow \qunifrm_{\rm
  trans,\, tot\, bd}$, where $\qunifrm_{\rm trans,\, tot\, bd}$
denotes the full subcategory of $\qunifrm$ consisting of those
quasi-uniform frames that are transitive and totally bounded.

Before proceeding, we show a couple of technical results concerning
$C$-ideals of the form $E_r$. The first one is a straightforward
  computation.
\begin{lemma}\label{l:11}
  Let $h: K_1 \to K_2$ be a frame homomorphism. Then, for every $r \in
  K_1$, we have $(h \oplus h)(E_{r}) \subseteq E_{{h(r)}}$. The
  converse inclusion holds if~$r$ is complemented.
\end{lemma}
% \begin{proof}
%   Given $r \in K_1$, we may compute:
%   \begin{align*}
%     (h \oplus h)(E_{r})
%     & \just {\eqref{eq:2}}= \bigvee \{h(x)\oplus h(y) \mid x \leq r \text{ or } y \leq
%       r^*\}
%       = (h(r) \oplus 1)  \vee (1 \oplus h(r^*))
%     \\ & \subseteq (h(r) \oplus 1)  \vee (1 \oplus h(r)^*)
%          =  E_{{h(r)}},
%   \end{align*}
%   where the only inclusion follows from the inequality $h(r^*) \leq
%   h(r)^*$ (recall~\eqref{eq:5}).
% %
%   Now, if $r$ is complemented, we have $h(r^*) = h(r)^*$, and the
%   above inclusion becomes an equality.
% \end{proof}

\begin{lemma}\label{l:5}
  Let $K$ be a frame and $r, r_1, \dots, r_n \in K$ be complemented
  elements. If $\bigcap_{i = 1}^n E_{r_i} \subseteq E_r$, then $r$ is
  a lattice combination of the elements of $\{r_1, \dots, r_n\}$.
\end{lemma}
\begin{proof}
  As before, we use $[n]$ to denote the set $\{1, \dots, n\}$ and, for
  every $P \subseteq [n]$ we let $r_P := \bigwedge_{i \in P} r_i$ and
  $\overline{r}_P := \bigwedge_{i \notin P} r_i^*$. Since, for every
  $P \subseteq [n]$, we have $(r_P, \overline{r}_P) \in \bigcap_{i =
    1}^nE_{r_i}$, it follows from the hypothesis that $(r_P,
  \overline{r}_P) \in E_r$, that is, $r_P \leq r$ or $\overline{r}_P
  \leq r^*$. Thus, by~\eqref{eq:5}, for every $P \subseteq [n]$, we
  have
  \begin{equation}
    \label{eq:10}
    r_P \leq r\qquad \text{or}\qquad\overline{r}_P \wedge r = 0.
  \end{equation}
  On the other hand, since each $r_i$ is complemented, we have
  \[r = r \wedge \bigwedge_{i = 1}^n (r_i \vee r_i^*) = r \wedge
    \bigvee \{r_P \wedge \overline{r}_P \mid P \subseteq [n]\} =
    \bigvee \{r \wedge r_P \wedge \overline{r}_P \mid P \subseteq
    [n]\}.\]
  Using~\eqref{eq:10}, we may then conclude that $r = \bigvee \{r_P
  \mid P \subseteq [n], \, r \wedge \overline{r}_P \neq 0\}$, that is,
  $r$ is a lattice combination of the elements of $\{r_1, \dots,
  r_n\}$.
\end{proof}

Let us now define the functor $E$ on the morphisms.  If $h:(L, S) \to
(M, T)$ is a morphism of Frith frames then, by Corollary~\ref{c:1},
$h$ uniquely extends to a frame homomorphism $\overline h: \cC_SL \to
\cC_TM$. The fact that $\overline{h}$ defines a quasi-uniform
homomorphism $\overline{h}: (\cC_S L, \cE_S) \to (\cC_TM, \cE_{T})$ is
a consequence of the following more general result, which can be
  proved using Lemmas~\ref{l:11} and~\ref{l:5}.

\begin{proposition}\label{p:3}
  Let $K, N$ be two frames, and $R\subseteq K$ and $U \subseteq N$ be
  sublattices of complemented elements. Let also $g: K \to N$ be a
  frame homomorphism. Then,
  \begin{enumerate}
  \item\label{item:18} $g[R] \subseteq U$ if and only if $(g \oplus
    g)[\cE_R] \subseteq \cE_U$;
  \item\label{item:19} $U \subseteq g[R]$ if and only if $\cE_U$ is
    contained in the filter generated by $(g \oplus g)[\cE_R]$.
  \end{enumerate}
\end{proposition}
% \begin{proof}
%   We start by proving~\ref{item:18}. Suppose that $g[R] \subseteq U$
%   and let $r \in R$. Since $r$ is complemented, by Lemma~\ref{l:11},
%   we have $(g \oplus g)(E_r) = E_{g(r)}$, which belongs
%   to~$\cE_U$. Conversely, given $r \in R$, we have $E_{g(r)} = ( g
%   \oplus g)(E_r) \in \cE_U$ and so, there exist $u_1, \dots, u_n \in
%   U$ such that $E_{g(r)} \supseteq E_{u_1} \cap \dots \cap
%   E_{u_n}$. By Lemma~\ref{l:5}, and since $U$ is a sublattice of $N$,
%   it follows that $g(r) \in U$.

%   For proving~\ref{item:19}, we first assume that $U\subseteq
%   g[R]$. Then, given $u \in U$, there exists some $r \in R$ such that
%   $g(r) = u$ and, for such an~$r$, we have $E_u = E_{g(r)} = (g \oplus
%   g)(E_r)$. Thus, $\cE_U$ is contained in the filter generated by $(g
%   \oplus g)[\cE_R]$. Conversely, if $\cE_U$ is contained in the filter
%   generated by $(g \oplus g)[\cE_R]$ then, for every $u \in U$, there
%   are some $r_1, \dots, r_n \in R$ satisfying
%   \[E_{u} \supseteq ({g}\oplus {g})(E_{r_1}) \cap \dots \cap ({g}
%     \oplus {g})(E_{r_n}).\] Then, by Lemma~\ref{l:11} we have
%   \[E_u \supseteq E_{g(r_1)} \cap \dots \cap E_{g(r_n)},\]
%   and by Lemma~\ref{l:5}, $u$ is a lattice combination of the elements
%   of $\{g(r_1), \dots, g(r_n)\}$ and thus, it belongs to $g[R]$ as
%   required.
% \end{proof}

As a consequence, we have an embedding $E: \ffrm \hookrightarrow
\qunifrm_{\rm trans,\, tot\, bd}$ defined by $E(L, S) = (\cC_SL,\
\cE_S)$ and $E(h) = \overline{h}$. In fact, Proposition~\ref{p:3} also
implies that $E$ is full.

\begin{proposition}\label{p:9}
  There is a full embedding $E: \ffrm \hookrightarrow \qunifrm_{\rm
    trans,\, tot\, bd}$.
\end{proposition}

Our next goal is to show that the embedding $E:\ffrm \hookrightarrow
\qunifrm_{\rm trans,\, tot\, bd}$ is a coreflection
(cf. Theorem~\ref{t:4}). We will need the following technical result:
\begin{lemma}\label{l:1}  Let $K$ be a
  frame, and $R_1, R_2 \subseteq K$ be subsets of complemented
  elements. Then, $\cE_{R_1} = \cE_{R_2}$ if and only if $\langle
  R_1\rangle_\dlat = \langle R_2\rangle_\dlat$.
\end{lemma}
\begin{proof}
  For $i = 1,2$, we let $R_i'$ denote the lattice generated
  by~$R_i$. We first argue that $\cE_{R_i} = \cE_{R_i'}$, which
  implies that $\cE_{R_1} = \cE_{R_2}$ whenever $R_1$ and $R_2$
  generate the same sublattice of~$K$.  Since $R_i \subseteq R_i'$, we
  clearly have $\cE_{R_i} \subseteq \cE_{R_i'}$. For the reverse
  inclusion, it suffices to observe that for every $r, r' \in R_i$,
  the entourage $E_{r}\cap E_{r'}$ is contained both in $E_{r \wedge
    r'}$ and in $E_{r \vee r'}$. Let us prove the converse
  implication. We suppose that $\cE_{R_1} = \cE_{R_2}$ and we let $r
  \in R_1$. Then, $E_r \in \cE_{R_1}$ and so, there are some $r_1,
  \dots, r_n \in R_2$ such that $E_r \supseteq E_{r_1} \cap \dots \cap
  E_{r_n}$. By Lemma~\ref{l:5}, this yields $r \in \langle R_2
  \rangle_\dlat$ and thus, $\langle R_1\rangle_\dlat \subseteq \langle
  R_2\rangle_\dlat$. By symmetry, we have $\langle R_1\rangle_\dlat =
  \langle R_2\rangle_\dlat$ as required.
\end{proof}

We now show that every transitive and totally bounded quasi-uniformity
$\cE$ on a frame $K$ is of the form $\cE_{R}$ for a suitable bounded
sublattice $R$ of~$K$. We introduce the following notation:
\begin{definition}\label{sec:rke}
  Given a transitive and totally bounded quasi-uniform frame $(K,
  \mathcal E)$, we will denote by $R_{(K, \, \cE)}$ the set of all
  complemented elements $r$ of $K$ such that $E_r \in \cE$.
\end{definition}
The proof of the next result is inspired by the proof of an
unpublished result for Pervin spaces, which is due to Gehrke,
Grigorieff, and Pin.
\begin{proposition}\label{p:8}
  Let $(K, \mathcal E)$ be a transitive and totally bounded
  quasi-uniform frame. Then, $R_{(K, \, \cE)}$ is a sublattice of~$K$
  satisfying $\cE = \cE_{R_{(K,\, \cE)}}$.
\end{proposition}
\begin{proof}
  Let $(K, \cE)$ be a transitive and totally bounded quasi-uniform
  frame. By Lemma~\ref{l:1}, we have that $R_{(K,\, \cE)}$ is a
  sublattice of~$K$,
  and by definition of~$R_{(K,\, \cE)}$, we also have $\cE_{R_{(K,\,
      \cE)}} \subseteq \cE$. Since~$\cE$ is a transitive
  quasi-uniformity, in order to show the converse inclusion, it
  suffices to show that $\cE_{R_{(K,\, \cE)}}$ contains every
  transitive entourage of~$\cE$.

  Let us fix a transitive entourage $E \in \cE$.  Since $\cE$ is
  totally bounded, there exists a finite cover $C$ of $K$ such that
  $\bigvee_{x \in C} x \oplus x \subseteq E$. Moreover, since $E$ is
  transitive, we may assume without loss of generality that $C$ is a
  partition. Indeed, suppose otherwise, say $x \wedge y \neq 0$ for
  some distinct $x, y \in C$. Then, since $(x,x), (y, y) \in E$
  implies, by~\ref{item:J1}, that $(x, x\wedge y), (x \wedge y, y) \in
  E$, by transitivity of~$E$, we have that $(x, y)$ belongs to~$E$,
  and by~\ref{item:J3}, it follows that $(x, x \vee y) \in
  E$. Similarly, we can show that $(x \vee y, x) \in E$. Using again
  transitivity of~$E$, we may conclude that $(x \vee y, x\vee y)$
  belongs to~$E$. Now, for each $x \in C$, let us denote
  \begin{equation}
    r_x := \bigvee \{y \in C \mid (x, y) \notin E\}.\label{eq:7}
  \end{equation}
  Since $C$ is a partition of~$M$, each $r_x$ is complemented with
  complement given by
  \begin{equation}
    r_x ^*= \bigvee \{z \in C \mid (x, z) \in
    E\}.\label{eq:1}
  \end{equation}
  We now show the following equality:
  \begin{equation}
    E = \bigcap \{E_{r_x} \mid x \in C\}\label{eq:6}
  \end{equation}
  Let $(a,b) \in K \times K$ be such that $a$ and $b$ are both
  non-zero, or else, $(a,b)$ belongs to both sides
  of~\eqref{eq:6}. Suppose that $(a, b) \in E$ and let $x \in
  C$. There are two cases: (a) $a \leq r_x$ and (b) $a \nleq
    r_x$. In the first case, $a$ is trivially in $E_{r_x}$. Let us
    assume that $a \nleq r_x$. Then, for showing $(a, b) \in E_{r_x}$,
    we need to show that $b \leq r_x^*$. Since $r_x$ is
  complemented, by~\eqref{eq:5}, having $a \nleq r_x$ is equivalent to
  having $a \wedge r_x^* \neq 0$. Therefore, by~\eqref{eq:1}, there
  exists some $z \in C$ such that $a \wedge z \neq 0$ and $(x, z) \in
  E$. On the other hand, again by~\eqref{eq:5}, $b \leq r_x^*$ if and
  only if $b \wedge r_x = 0$ and, by~\eqref{eq:7}, $b \wedge r_x = 0$
  if and only if $(x, y) \in E$ for every $y \in C$ satisfying $b
  \wedge y \neq 0$. So, we let $y \in C$ be such that $b \wedge y \neq
  0$. Since $(x, z)$, $(a, b)$, and $(y, y)$ belong to $E$ and $E$ is
  a downset, we have that $(x,a \wedge z)$, $(a \wedge z, b \wedge
  y)$, and $(b \wedge y, y)$ also belong to~$E$. By transitivity
  of~$E$, and since $a \wedge z, b \wedge y \neq 0$, it follows that
  $(x, y) \in E$. This shows that $b \leq r_x^*$ as required.
  Conversely, suppose that $(a,b)$ belongs to $E_{r_x}$, for every $x
  \in C$. Since $C$ is a partition of $K$, we have $a = \bigvee\{a
  \wedge x \mid x \in C\}$ and $b = \bigvee \{b \wedge y \mid y \in
  C\}$. Therefore, by~\ref{item:J2} and~\ref{item:J3}, in order to
  show that $(a,b) \in E$, it suffices to show that $(a \wedge x, b
  \wedge y) \in E$ for every $x, y \in C$ satisfying $a \wedge x \neq
  0$ and $b \wedge y \neq 0$. Since $x \leq r_x^*$, if $a \wedge x
  \neq 0$ then we also have $a \wedge r_x^* \neq 0$, that is, $a \nleq
  r_x$. Since $(a, b) \in E_{r_x}$, this implies $b \leq r^*_x$, that
  is, $b \wedge r_x = 0$. Finally, by~\eqref{eq:7}, if $b \wedge r_x =
  0$, then $(x, y) \in E$ whenever $b \wedge y \neq 0$. Since $E$ is a
  downset, we have $(a \wedge x, b \wedge y) \in E$ and this finishes
  the proof of~\eqref{eq:6}.

  Finally, \eqref{eq:6} implies, on the one hand, that each $r_x$
  belongs to $R_{(K, \cE)}$ (because $r_x$ is complemented, $E \in
  \cE$ and $E = \bigcap_{x\in C} E_{r_x} \subseteq E_{r_x}$) and, on
  the other hand, that $E \in \cE_{R_{(K, \cE)}}$ (because $E =
  \bigcap_{x\in C} E_{r_x}$ and each $r_x$ belongs to $R_{(K,
    \cE)}$). Therefore, $\cE = \cE_{R_{(K, \cE)}}$ as required.
\end{proof}
In particular, using Theorem~\ref{t:1}, we have the following:
\begin{corollary}
  Every frame admitting a transitive and totally bounded
  quasi-uniformity is zero-dimensional.
\end{corollary}

\begin{proposition}\label{p:18}
  Given a transitive and totally bounded quasi-uniform frame $(K,
  \cE)$, we let $L$ be the subframe of~$K$ generated by~$R_{(K,\,
    \cE)}$ and we let $e:L \hookrightarrow K$ be the corresponding
  embedding, so that we have a Frith frame $(L, S)$, where $S$ is such
  that $e[S] = R_{(K, \, \cE)}$. Then, the embedding $e$ extends to a
  dense extremal epimorphism $\gamma_{(K, \,\cE)}: (\cC_SL, \cE_S)
  \twoheadrightarrow (K, \cE)$ of quasi-uniform frames.
\end{proposition}
\begin{proof}
  Since the elements of~$e[S] = R_{(K,\, \cE)}$ are complemented
  in~$K$, by Proposition~\ref{p:10}, the embedding $e : L
  \hookrightarrow K$ uniquely extends to a frame homomorphism
  $\gamma_{(K, \,\cE)}: \cC_SL \to K$. Moreover, $\gamma_{(K, \,\cE)}$
  is surjective: since frame homomorphisms preserve complemented
  elements, we have $R_{(K,\, \cE)} \cup R_{(K,\, \cE)}^* =
  \gamma_{(K, \,\cE)}\ [\{\nabla_s, \Delta_s\}_{s \in S}]$ and, by
  Theorem~\ref{t:1} and Proposition~\ref{p:8}, it follows that
  \[K = \langle R_{(K,\, \cE)} \cup R_{(K,\, \cE)}^* \rangle_\Frm =
    \gamma_{(K, \,\cE)}[\cC_SL].\]
  Since, by Lemma~\ref{l:11}, we have $(\gamma_{(K, \,\cE)} \oplus
  \gamma_{(K, \,\cE)})(E_{\nabla_s}) = E_{ \gamma_{(K, \,\cE)}(s)} =
  E_{e(s)}$ for every $s \in S$, and by Proposition~\ref{p:8}, $\cE =
  \cE_{R_{(K,\, \cE)}}$, we may conclude that $\gamma_{(K, \,\cE)}$
  induces a quasi-uniform homomorphism $\gamma_{(K, \,\cE)}: (\cC_SL,
  \cE_S) \twoheadrightarrow (M, \cE)$, which is an extremal
  epimorphism.  It remains to show that $\gamma_{(K, \, \cE)}$ is
  dense. Let $s_1, s_2 \in S$ be such that $\gamma_{(K, \,
    \cE)}(\nabla_{s_1} \wedge \Delta_{s_2}) = 0$, or equivalently,
  $e(s_1) \wedge e(s_2)^* = 0$. Since $e(s_2) \in R_{(K, \, \cE)}$ is
  complemented, by~\eqref{eq:5}, it follows that $e(s_1) \leq
  e(s_2)$. Since $e$ is an embedding, we have $s_1 \leq s_2$ and we
  obtain $\nabla_{s_1} \wedge \Delta_{s_2} = 0$ as required.
\end{proof}

Finally, we may use Proposition~\ref{p:18} to show that $E: \ffrm
\hookrightarrow \qunifrm_{\rm trans,\, tot\, bd}$ is a coreflection.
\begin{theorem}\label{t:4}
  The category of Frith frames is coreflective in the category of
  transitive and totally bounded quasi-uniform frames.
\end{theorem}
\begin{proof} For a transitive and totally bounded quasi-uniform frame
  $(K, \cE)$, we let $\gamma_{(K, \cE)} : (\cC_SL,
  \cE_S)\twoheadrightarrow (K, \cE)$ be the dense extremal epimorphism
  of Proposition~\ref{p:18}, and $e: L \hookrightarrow K$ denote the
  embedding of~$L$ into~$K$. To show that $E: \ffrm \hookrightarrow
  \qunifrm_{\rm trans,\, tot\, bd}$ is a coreflection, it suffices to
  show that, for every Frith frame $(M, T)$ and every quasi-uniform
  homomorphism $h: (\cC_TM, \cE_T) \to (K, \cE)$, there exists a
  unique morphism of Frith frames $g: (M, T) \to (L, S)$ such that the
  following diagram commutes:
    \begin{center}
      \begin{tikzpicture}
      [->, node distance = 30mm] \node (A) {$(\cC_SL, \cE_S)$};
      \node[right of = A] (B) {$(K, \cE)$}; \node[above of = B, yshift
      = -10mm] (C) {$(\cC_TM, \cE_T)$};
      \draw[->>] (A) to node[below, ArrowNode] {$\gamma_{(K, \cE)}$}
      (B); \draw (C) to node[right, ArrowNode] {$h$} (B);
      \draw[dashed] (C) to node[left, yshift = 2mm, ArrowNode] {$Eg =:
        \overline{g}$}(A);
    \end{tikzpicture}
  \end{center}
  If such a morphism~$g$ exists then, in particular, we must have
  $\gamma_{(K, \cE)}\circ \overline{g}(\nabla_a) = h(\nabla_a)$ for
  every $a \in M$. Since $\overline{g}$ and $\gamma_{(K, \cE)}$ are
  extensions of $g$ and $e$, respectively, this amounts to having $e
  \circ g (a) = h(\nabla_a)$ for every $a \in M$. Since by
  Proposition~\ref{p:3}\ref{item:18}, we have $h[T] \subseteq R_{(K,
    \cE)}$ and thus, $h[M] \subseteq \langle R_{(K, \cE)}\rangle_\Frm
  = e[L]$ and $e$ is an embedding, there is a unique morphism of Frith
  frames $g: (M, T) \to (L, S)$ satisfying $e \circ g (a) =
  h(\nabla_a)$ for every $a \in M$. Finally, to see that
  $\overline{g}$ makes the above diagram commute, it suffices to
  observe that, for every $t \in T$, the following equalities hold:
  \[\gamma_{(K, \cE)} \circ \overline{g}(\Delta_t) = \gamma_{(K,
      \cE)} (\Delta_{g(t)}) \stackrel{\tiny (*)}{=} (e\circ g(t))^* = h(\nabla_t)^*=
    h(\Delta_t),\]
  where the equality marked by~$(*)$ holds because $g(t)$ belongs
  to~$S$.
\end{proof}
Note that, unlike what happens in the point-set framework, where we
have an equivalence between Pervin spaces and transitive and totally
bounded quasi-uniform spaces (see~\cite{pin17}), the categories of
Frith frames and of transitive and totally bounded quasi-uniform
frames are not equivalent. This is because, in general, the dense
extremal epimorphism $\gamma_{(K,\, \cE)}: (\cC_SL, \cE_S)
\twoheadrightarrow (K, \cE)$ from Proposition~\ref{p:18} is not an
isomorphism.

\begin{example}
  Let $X$ be a topological space such that the congruence frame of
  its frame of opens is not spatial
  (see~\cite[Theorem~3.4]{niefield87} for a characterization of the
  frames whose congruence frame is not spatial).  We let $R$ be the
  frame of opens of~$X$ and $K$ be the frame of opens of its Skula
  topology (that is, the topology generated by the elements of $R$ and
  their complements).  Then, the underlying map of $\gamma_{(K,\,
    \cE_R)}: (\cC R, \cE_R) \twoheadrightarrow (K, \cE)$ is the unique
  frame homomorphism extending $R \hookrightarrow K$.  This cannot be
  an isomorphism because $K$ is spatial and $\cC R$ is not.
  A concrete example is given by taking for $X$ the real line $\mathbb
  R$ equipped with the Euclidean topology: since the Booleanization of
  $\Omega(\mathbb R)$ is a pointless nontrivial sublocale, by the
  characterization of~\cite{niefield87}, its congruence frame is not
  spatial. Note that this example illustrates a wider class of
  transitive and totally bounded quasi-uniform frames that are not in
  the image of~$E$, namely, those of the form $(K, \cE_R)$ where $K$
  and $R$ as in Theorem~\ref{t:1} are such that $R$ is a subframe of
  $K$ and $K$ is not isomorphic to the congruence frame of~$R$
  (see~\cite{MR4074752} for a characterization of those).\footnote{For
    the reader who is familiar with biframes, this is the same as
    saying that every strictly 0-dimensional biframe $(K_0, K_1, K_2)$
    that is not the congruence biframe of its first part gives rise to
    such an example.}
\end{example}

On the other hand, in the case where~$(K, \cE)$ is such that the
  elements of $S = R_{(K,\, \cE)}$ are complemented in~$L$, one may
  easily check that $\gamma_{(K, \, \cE)}$ is an isomorphism. As we
shall see in the next section, this happens if $\cE$ is a
uniformity~(cf. Proposition~\ref{l:9}) and so, the category of
transitive and totally bounded uniform frames may be nicely
represented by a suitable full subcategory of Frith frames
(cf. Corollary~\ref{c:5}). A similar conclusion may be taken if we
restrict to those transitive and totally bounded quasi-uniform frames
that are complete (cf. Corollary~\ref{c:8}): if $(K, \cE)$ is
complete, then $\gamma_{(K, \cE)}$ has to be an isomorphism.
%

% Note that, unlike what happens in the point-set framework, where we
% have an equivalence between Pervin spaces and transitive and totally
% bounded quasi-uniform spaces (see~\cite{pin17}), the categories of
% Frith frames and of transitive and totally bounded quasi-uniform
% frames are not equivalent. This is because, in general, the dense
% extremal epimorphism $\gamma_{(K,\, \cE)}: (\cC_SL, \cE_S)
% \twoheadrightarrow (K, \cE)$ from Proposition~\ref{p:18} is not an
% isomorphism. Intuitively, the reason behind this phenomenon is that
% there are several ways in which~$L := \langle R_{(K, \,
%   \cE)}\rangle_{\Frm}$ may be extended so that every element of~$S :=
% R_{(K, \cE)}$ becomes complemented, and $\cC_SL$ is just one of them
% that may not reflect how the complements in $K$ of the elements
% of~$R_{(K,\, \cE)}$ behave. Accordingly, if~$(K, \cE)$ is such that
% the elements of $S = R_{(K,\, \cE)}$ are already complemented in~$L$,
% then $\gamma_{(K, \, \cE)}$ is an isomorphism. As we shall see in the
% next section, this happens if $\cE$ is a
% uniformity~(cf. Proposition~\ref{l:9}) and so, the category of
% transitive and totally bounded uniform frames may be nicely
% represented by a suitable full subcategory of Frith frames
% (cf. Corollary~\ref{c:5}). A similar conclusion may be taken if we
% restrict to those transitive and totally bounded quasi-uniform frames
% that are complete (cf. Corollary~\ref{c:8}): if $(K, \cE)$ is
% complete, then $\gamma_{(K, \cE)}$ has to be an isomorphism.

%%%%%%%%%%%%%%%%%%%%%%%%%%%%%%%%%%%%%%%%%%%%%%%%%%%%%%%%%%%%%%%%%%%%%%
\section{Symmetric Frith frames and uniform frames}\label{sec:4}

We say that a Frith frame $(L,B)$ is \emph{symmetric} if $B$ is a
Boolean algebra, and we let $\sffrm$ denote the full subcategory of
$\ffrm$ whose objects are the symmetric Frith frames. The relevance of
symmetric Frith frames lies in the fact that they exactly capture
those transitive and totally bounded quasi-uniform frames that are
actually uniform.

\begin{proposition}\label{l:9}
  Let $(K, \cE)$ be a transitive and totally bounded quasi-uniform
  frame. Then, $\cE$ is a uniformity if and only if $R_{(K,\, \cE)}$
  is a Boolean algebra. In particular, for every Frith frame $(L, S)$,
  we have that $(L, S)$ is symmetric if and only if $(\cC_SL, \cE_S)$
  is a uniform frame.
\end{proposition}
\begin{proof}
  The first claim is a straightforward consequence of
  Proposition~\ref{p:8} and of the equality $E_{r}^{-1} = E_{r^*}$,
  which holds whenever $r$ is complemented. For the second statement,
  we only need to observe that, by Proposition~\ref{p:8}, we have
  $\cE_S = \cE_{R_{(\cC_SL,\, \cE_S)}}$ and thus, by Lemma~\ref{l:1},
  we have $R_{(\cC_SL,\, \cE_S)}= \{\nabla_s \mid s \in S\}$ which, by
  Lemma~\ref{basic}, is a lattice isomorphic to $S$. Therefore, by the
  first claim, $(\cC_SL, \cE_S)$ is a uniform frame if and only if $S$
  is a Boolean algebra. 
\end{proof}
In the following, we let $\unifrm_{\rm trans,\, tot\, bd}$ denote the
full subcategory of $\qunifrm$ formed by the transitive and totally
bounded uniform frames.
\begin{corollary}\label{c:5}
  The coreflection $E: \ffrm \hookrightarrow\qunifrm_{\rm trans,\,
    tot\, bd}$ restricts and co-restricts to an isomorphism $E':
  \sffrm \hookrightarrow \unifrm_{\rm trans,\, tot\, bd}$.
\end{corollary}
\begin{proof}
  By Proposition~\ref{l:9}, the functor $E$ restricts and co-restricts
  to a functor $E': \sffrm \hookrightarrow \unifrm_{\rm trans,\, tot\,
    bd}$. Let $(K, \cE)$ be a transitive and totally bounded uniform
  frame. By Proposition~\ref{l:9}, $R_{(K, \cE)}$ is a Boolean algebra
  and so, by Theorem~\ref{t:1} and Proposition~\ref{p:8}, $K$ is
  generated, as a frame, by $R_{(K, \cE)}$. Therefore, the pair $(K,
  R_{(K, \cE)})$ is a symmetric Frith frame. Moreover, by
  Propositions~\ref{p:3}\ref{item:18} and~\ref{p:8}, if $h:(K, \cE)
  \to (M, \cF)$ is a homomorphism of transitive and totally bounded
  uniform frames, then $h$ induces a morphism of (symmetric) Frith
  frames $h: (K, R_{(K, \cE)}) \to (M, R_{(M, \cF)})$. Hence, the
  assignment $(K, \cE) \mapsto (K, R_{(K, \cE)})$ yields a
  well-defined functor $\gamma': \unifrm_{\rm trans,\, tot\, bd} \to
  \sffrm$. Finally, one can easily show that $E'$ and $\gamma'$ are
  mutually inverse.
\end{proof}

We will now show that the category $\sffrm$ is both reflective and
coreflective in $\ffrm$.
Let us start by coreflectivity. 

\begin{proposition}\label{p:19}
  The category $\sffrm$ is a full coreflective subcategory of $\ffrm$.
\end{proposition}
\begin{proof}
  Let $(L,S)$ be a Frith frame. We consider the subset $C \subseteq S$
  consisting of those elements having a complement in $S$, and we let
  $N$ be the subframe of~$L$ generated by~$C$. Observe that $C$ is a
  Boolean subalgebra of $N$, so that we have a symmetric Frith frame
  $b(L, S) := (N, C)$. Moreover, the embedding $N \hookrightarrow L$
  defines a homomorphism of Frith frames $e_{(L, S)}:b(L,
  S)\hookrightarrow (L,S)$. To complete the proof, we only need to
  show that, for every symmetric Frith frame $(M,B)$ and for every
  morphism $h:(M,B)\ra (L,S)$ there is a unique morphism
  $\widetilde{h}:(M,B)\ra b(L, S)$ making the following diagram
  commute:
  \begin{center}
    \begin{tikzpicture}
      [->, node distance = 30mm] \node (A) {$b(L, S)$}; \node[right of
      = A] (B) {$(L, S)$}; \node[above of = B, yshift = -10mm] (C) {$(M, B)$};
      \draw[right hook->] (A) to node[below, ArrowNode] {$e_{(L, S)}$}
      (B); \draw (C) to node[right, ArrowNode] {$h$} (B);
      \draw[dashed] (C) to node[left, yshift = 2mm, ArrowNode]
      {$\widetilde h$}(A);
    \end{tikzpicture}
  \end{center}
  That is, we need to show that $h$ co-restricts to a morphism of
  Frith frames $\widetilde h: (M, B) \to b(L, S)$. This is indeed the
  case because, since frame homomorphisms preserve pairs of
  complemented elements, $h[B] \subseteq S$, and $B$ is a Boolean
  algebra, every element of $h[B]$ is complemented in~$S$, that is, we
  have $h[B] \subseteq C$.
\end{proof}
Combining Theorem~\ref{t:4}, Corollary~\ref{c:5}, and
Proposition~\ref{p:19}, we also have the following:
\begin{corollary}
  The category $\unifrm_{\rm trans,\, tot\, bd}$ is coreflective in
  $\qunifrm_{\rm trans,\, tot\, bd}$.
\end{corollary}

Let us now prove that $\unifrm_{\rm trans,\, tot\, bd}$ is reflective
subcategory of~$\qunifrm_{\rm trans,\, tot\, bd}$. Given a Frith frame $(L,S)$, we let
$\overline S$ denote the sublattice of $\cC_S L$ generated by the
elements of the form $\nabla_s$ together with their complements. As
complemented elements of a frame are closed under finite meets and
finite joins, the lattice $\overline S$ is a Boolean
algebra. Moreover, since $L$ is generated by $S$, we have that
$\cC_SL$ is generated by~$\overline{S}$ and so $(\cC_SL, \overline S)$
is a Frith frame. We may then define a functor $\fsym:\ffrm\ra \sffrm$
as follows. For a Frith frame $(L,S)$ we set $\fsym(L,S) := (\cC_S
L,\overline S)$; and for a morphism of Frith frames $h: (L,S)\ra
(M,T)$ we set $\fsym(h) := \overline h$, where $\overline h$ is the
unique extension of~$h$ to a frame homomorphism $\overline{h}: \cC_SL
\to \cC_TM$ (recall Corollary~\ref{c:1}). In particular, we have
$\overline{h}(\nabla_s)=\nabla_{h(s)}$ and $\overline{h}(\Delta_s)=
\Delta_{h(s)}$ for every $s\in S$ and thus, $\overline{h}: (\cC_SL,
\overline{S}) \to (\cC_TM, \overline{T})$ is a morphism of Frith
frames and $\fsym$ is a well-defined functor. We shall refer to
$\fsym(L, S)$ as the \emph{symmetrization} of $(L, S)$.

\begin{proposition}\label{p:25}
  The full subcategory of symmetric Frith frames is reflective in
  $\ffrm$.
\end{proposition}
\begin{proof}
  We first observe that if $(L,S)$ is a Frith frame, then it embeds in
  its symmetrization via $\nabla: L \hookrightarrow \cC_SL$. Let $(M,
  B)$ be a symmetric Frith frame, and $h:(L,S)\ra (M,B)$ be a morphism
  in $\ffrm$. We need to show that there is a unique morphism
  $\widetilde{h}:(\cC_S L, \overline{S})\ra (M,B)$ making the
  following diagram commute:
   \begin{center}
    \begin{tikzpicture}
      [->, node distance = 30mm]
      \node (A) {$(L, S)$};
      \node[right of = A] (B) {$(\cC_S L, \overline{S})$};
      \node[below of = B, yshift = 10mm] (C) {$(M, B)$};
      \draw[right hook->] (A) to node[above, ArrowNode] {$\nabla$}
      (B); \draw[dashed] (B) to node[right, ArrowNode] {$\widetilde
        h$} (C); \draw (A) to node[left, yshift = -1mm, ArrowNode]
      {$h$}(C);
    \end{tikzpicture}
  \end{center}
  Since $h[S] \subseteq B$ consists of complemented elements of~$M$,
  by Proposition~\ref{p:10}, there is a unique frame map
  $\widetilde{h}:\cC_S L\ra M$ making the above triangle
  commute. Hence, we only need to show that $\widetilde
  h[\,\overline{S}\,] \subseteq B$. This is indeed the case because,
  for every $s \in S$, we have $\widetilde{h}(\nabla_s)=h(s)$ and
  $\widetilde{h}(\Delta_s)=h(s)^*$, and $B$ is closed under taking
  complements.
\end{proof}

We now argue that $\fsym$ is a restriction of the usual reflection
of $\qunifrm$ onto $\unifrm$.

\begin{proposition}\label{p:2}
  The following diagram commutes:
  \begin{center}
    \begin{tikzpicture}
      \node (A) {$\ffrm$}; \node[right of = A, xshift = 30mm] (B)
      {$\sffrm$}; \node[below of = A] (C) {$\qunifrm$}; \node[below of
      = B] (D) {$\unifrm$};
      \draw (A) to node[ArrowNode,above] {$\fsym$} (B); \draw[right
      hook->] (B) to node[ArrowNode,right] {$E$} (D); \draw[right
      hook->] (A) to node[ArrowNode,left] {$E$} (C); \draw (C) to
      node[ArrowNode, below] {$\sym$} (D);
    \end{tikzpicture}
  \end{center}
\end{proposition}
\begin{proof}
  Let $(L, S)$ be a Frith frame. By definition, we have
  \[E \circ \fsym(L, S) = (\cC_SL, \cE_{\overline S}) \qquad
    \text{and}\qquad \sym \circ E(L, S) = (\cC_SL,
    \overline{\cE_S}),\]
  where $\overline{\cE_S}$ is the uniformity generated by
  $\{E_{\nabla_s}, E_{\nabla_s}^{-1} \mid s \in S\}$. Since
  $E_{\nabla_s}^{-1} = E_{\Delta_s}$, by definition of~$\overline{S}$
  and by Lemma~\ref{l:1}, $\cE_{\overline S}$ and $\overline{\cE_S}$
  are both the (quasi-)uniformity generated by $\{E_{\nabla_s},
  E_{\Delta_s} \mid s \in S\}$. Commutativity at the level of
  morphisms is trivial.
\end{proof}

Finally, we will see that $\fsym$ is the pointfree analogue of $\psym$
discussed in Section~\ref{sec:pervin}.

\begin{proposition}\label{p:23}
  There is an isomorphism $\alpha_{(L, S)}: \psym\circ
  \pt(L,S)\cong\pt\circ \fsym(L,S)$ for every Frith frame $(L,S)$.
\end{proposition}
\begin{proof}
  Let us define $\alpha_{(L,S)}: \pt(L) \to \pt(\cC_SL)$ by $p\mapsto
  \widetilde{p}$, where for every point $p\in\pt(L)$ the map
  $\widetilde{p}$ is the unique morphism such that $\widetilde p \circ
  \nabla = p$, as given by Proposition~\ref{p:10}. Since $L$ is a
  subframe of $\cC_SL$ and thus, every point of $\cC_SL$ restricts to
  a point of~$L$, this assignment is a bijection. Let us show that
  $\alpha_{(L,S)}$ defines an isomorphism of Pervin spaces
  $\alpha_{(L, S)}: \psym\circ \pt(L,S)\to\pt\circ \fsym(L,S)$. By
  Corollary~\ref{perviniso}, we need to show that the preimages of the
  elements of the lattice component of $\pt\circ\fsym(L,S)$ are
  exactly the elements of the lattice component of
  $\psym\circ\pt(L,S)$. Noting that the former lattice is generated by
  the elements of the form $\widehat \nabla_s$ and $\widehat \Delta_s$
  and the latter by those of the form $\widehat s$ and $(\widehat
  s)^c$ ($s \in S$), that is a consequence of having
  \[\alpha_{(L, S)}^{-1}(\widehat{\nabla_s}) = \widehat s \qquad
    \text{and}\qquad \alpha_{(L, S)}^{-1}(\widehat{\Delta_s}) =
    (\widehat s)^c,\] for every $s\in S$.
\end{proof}

\begin{theorem}\label{t:5}
  The following diagram commutes up to natural
  isomorphism.
  \begin{center}
    \begin{tikzpicture}[node distance = 15mm]
      \node (A) {$\ffrm^{\rm op}$}; \node [right of = A, xshift =
      25mm] (B) {$\sffrm^{\rm op}$}; \node [below of = A] (C)
      {$\perv$}; \node [below of = B] (D) {$\sperv$};
      \draw (A) to node[left, ArrowNode] {$\pt$} (C);
      \draw (C) to node[below, ArrowNode] {$\psym$}(D);
      \draw (B) to node[right, ArrowNode] {$\pt$} (D);
      \draw (A) to node[above, ArrowNode] {$\fsym$} (B);
    \end{tikzpicture}
  \end{center}
\end{theorem}
\begin{proof}
  We show that the family of isomorphisms $\{\alpha_{(L, S)} \mid (L,
  S) \text{ is a Frith frame}\}$ defined in Proposition~\ref{p:23}
  induces a natural transformation $\psym\circ \pt\implies\pt\circ
  \fsym$. Suppose that $h:(M,T)\ra (L, S)$ is a morphism of Frith
  frames, that is, $h$ is a morphism $(L, S) \to (M, T)$ in
  $\ffrm^{\rm op}$. Then, naturality of $\alpha$ amounts to commutativity
  of the following square:
   \begin{center}
    \begin{tikzpicture}
      \node (A) {$\pt(L)$}; \node[right of = A, xshift = 20mm] (B)
      {$\pt(\cC_SL)$}; \node[below of = A] (C) {$\pt(M)$}; \node[below
      of = B] (D) {$\pt(\cC_TM)$};
      \draw (A) to node[ArrowNode,above] {$\alpha_{(L,S)}$} (B); \draw
      (B) to node[ArrowNode,right] {$(-)\circ \overline h$} (D);
      \draw (A) to node[ArrowNode,left] {$(-)\circ h$} (C); \draw (C)
      to node[ArrowNode, below] {$\alpha_{(M,T)}$} (D);
    \end{tikzpicture}
  \end{center}
  Let $p\in \pt(L)$. Since, by definition, $\alpha_{(M, T)}(p \circ
  h)$ is the unique point of $\pt(\cC_TM)$ extending $p \circ h$, it
  suffices to show that $\alpha_{(L, S)} (p)\circ \overline{h}$
  extends $p \circ h$. That is indeed the case because, for every $a
  \in M$, we have the following:
  \[\alpha_{(L, S)}(p)\circ \overline{h} ( \nabla_a) = \alpha_{(L, S)}(p)(\nabla_{h(a)})
    \just {(*)}= p\circ h(a),\]
  where the equality $(*)$ follows from having that, by definition,
  $\alpha_{(L, S)}(p)$ is the unique extension of~$p$ to a point of
  $\cC_SL$.
\end{proof}

%%%%%%%%%%%%%%%%%%%%%%%%%%%%%%%%%%%%%%%%%%%%%%%%%%%%%%%%%%%%%%%%%%%%%%
\section{Completion of a Frith frame}\label{sec:7}

As mentioned in Section~\ref{sec:prelim}, complete (quasi-)uniform
frames may be equivalently characterized via dense extremal
epimorphisms or via Cauchy maps. Since the category of Frith frames
fully embeds into the category of quasi-uniform frames, there is a
natural notion of completion of a Frith frame. In this section we
explore it, both from the point of view of dense extremal epimorphisms
and of Cauchy maps. As the reader will notice, in this restricted
subcategory of quasi-uniform frames, the concepts involved become
surprisingly simple.

\subsection{Dense extremal epimorphisms}\label{sec:8}
We say that a symmetric Frith frame $(L, B)$ is \emph{complete} if
every dense extremal epimorphism $(M, C) \twoheadrightarrow (L, B)$
with $(M, C)$ symmetric is an isomorphism. More generally, a Frith
frame $(L, S)$ is \emph{complete} provided its symmetric
reflection~$\fsym(L, S)$ is complete.  As the reader may expect,
completeness of a Frith frame $(L, S)$ is equivalent to completeness
of the associated quasi-uniform frame $(\cC_S L, \cE_S)$
(cf. Proposition~\ref{p:1}). A \emph{completion} of $(L, S)$ is a
complete Frith frame $(M, T)$ together with a dense extremal
epimorphism $(M, T) \twoheadrightarrow (L, S)$.

Given a Frith frame $(L, S)$, by Corollary~\ref{c:7}, there is a
unique frame homomorphism $\idl(S) \to L$ extending the embedding
$S\hookrightarrow L$. Clearly, this frame homomorphism induces a dense
extremal epimorphism of Frith frames
\begin{equation}
  c_{(L, S)}: (\idl(S), S) \twoheadrightarrow (L, S), \quad J \mapsto
  \bigvee J.\label{eq:3}
\end{equation}

An immediate consequence of the definition of completeness is the
following:

\begin{lemma}\label{l:16}
  If $(L, S)$ is complete, then $c_{(\cC_SL, \overline
    S)}:(\idl(\overline S), \overline S) \to (\cC_SL, \overline S)$ is
  an isomorphism.
\end{lemma}
\begin{proof}
  This is simply because, by definition, $(L, S)$ is complete if and
  only if so is $\fsym(L, S) = (\cC_SL, \overline S)$ and $c_{(\cC_SL,
    \overline S)}$ is a dense extremal epimorphism.
\end{proof}

In particular, if $(L, S)$ is complete, then its symmetric reflection
$\fsym(L, S)$ is coherent. In turn, since $\fsym(L, S)$ is a
zero-dimensional Frith frame, by Lemma~\ref{l:13}, being coherent is
equivalent to being compact, and since $L$ is a subframe of $\cC_S L$,
we have $K(\cC_SL) \cap L \subseteq K(L)$. Now notice that $K(\cC_SL)
= \overline{S}$ by coherence of $\fsym(L, S)$ and Lemma~\ref{l:8}, and
thus $S\se K(\cC_S L)$. Therefore, $S\se K(\cC_S L)\cap L$ and this
implies that $S \subseteq K(L)$. This means that $(L, S)$ is coherent,
too. We have just proved the following:
\begin{equation}
  (L, S) \text{ complete} \implies \fsym(L, S) \text{ compact} \iff
  \fsym(L, S) \text{ coherent} \implies (L, S) \text{
    coherent}.\label{eq:4}
\end{equation}
From this, we may already show that our notion of \emph{completeness}
is consistent with usual completeness for (quasi-)uniform frames:
\begin{proposition}\label{p:1}
  Let $(L, S)$ be a Frith frame. Then, $(L, S)$ is complete if and
  only if $(\cC_SL, \cE_S)$ is complete.
\end{proposition}
\begin{proof}
  We first observe that it suffices to consider the case where $(L,
  S)$ is symmetric. Indeed, by definition, $(L, S)$ is complete if and
  only if $\fsym(L, S) = (\cC_SL, \overline S)$ is complete. On the
  other hand, by Proposition~\ref{l:2}, $(\cC_SL, \cE_S)$ is complete
  if and only if $(\cC_SL, \overline{\cE_S})$ is complete and, by
  Proposition~\ref{p:2}, we have $(\cC_SL, \overline{\cE_S}) = \sym
  \circ E(L, S) = E\circ \fsym(L, S) = (\cC_{S}L, \cE_{\overline
    S})$. Therefore, the claim holds if and only if, for every Frith
  frame $(L, S)$, completeness of $\fsym (L, S) = (\cC_SL,
  \overline{S})$ and of $E \circ \fsym(L, S) = (\cC_{S}L,
  \cE_{\overline S})$ are equivalent notions.

  Now, we let $(L, B)$ be a symmetric Frith frame. Suppose that $(L,
  B)$ is complete and let $h: (M, \cE) \twoheadrightarrow (L,
  \cE_{B})$ be a dense extremal epimorphism for some uniform frame
  $(M, \cE)$. Since $(L, B)$ is complete, by~\eqref{eq:4}, $L$ is
  compact.  Therefore, by Proposition~\ref{p:22}, $h$ is an
  isomorphism.
  Conversely, suppose that $(L, \cE_{B})$ is complete. By
  Proposition~\ref{p:3}, every dense extremal epimorphism $h: (M, C)
  \twoheadrightarrow (L, B)$ of symmetric Frith frames induces an
  extremal epimorphism $h:(M, \cE_{C}) \twoheadrightarrow (L,
  \cE_{B})$, which is clearly dense. Since $(L, \cE_{B})$ is complete,
  $h$ is one-to-one and, by Corollary~\ref{p:7}, it is an
  isomorphism. Thus, $(L, B)$ is complete as required.
\end{proof}

Before proceeding, we remark that a consequence of
Proposition~\ref{p:1} is that the categories $\bf C\ffrm$ of complete
Frith frames and $\bf C\qunifrm_{\rm trans,\, tot\, bd}$ of complete
transitive and totally bounded quasi-uniform frames are equivalent.
\begin{corollary}\label{c:8}
  The coreflection $E: \ffrm \hookrightarrow\qunifrm_{\rm trans,\,
    tot\, bd}$ restricts and co-restricts to an equivalence of
  categories $E'': \bf C\ffrm \to \bf C\qunifrm_{\rm trans,\, tot\,
    bd}$.
\end{corollary}
\begin{proof}
  By Proposition~\ref{p:1}, the functor $E: \ffrm
  \hookrightarrow\qunifrm_{\rm trans,\, tot\, bd}$ restricts and
  co-restricts to a functor $E'': \bf C\ffrm \hookrightarrow \bf
  C\qunifrm_{\rm trans,\, tot\, bd}$. Since $\bf C\ffrm$ and $\bf
  C\qunifrm_{\rm trans,\, tot\, bd}$ are, respectively, full
  subcategories of $\ffrm$ and of $\qunifrm_{\rm trans,\, tot\, bd}$,
  by Proposition~\ref{p:9}, $E''$ is a full embedding. Finally, let
  $(K, \cE)$ be a transitive and totally bounded quasi-uniform
  frame. If $(K, \cE)$ is complete, then the dense extremal
  epimorphism $\gamma_{(K, \cE)}: (\cC_SL, \cE_S) \twoheadrightarrow
  (K, \cE)$ of Proposition~\ref{p:18} has to be an isomorphism and,
  again by Proposition~\ref{p:1}, $(L, S)$ is complete. Therefore,
  $(K, \cE) \cong E''(L, S)$ and $E''$ is an equivalence of
  categories.
\end{proof}

Our next goal is to show that all the statements in~\eqref{eq:4} are
in fact equivalent. Since by Proposition~\ref{p:17} coherent Frith
frames are those of the form $(\idl(S), S)$, for some lattice~$S$, we
only need to prove that $(\idl(S), S)$ is always a complete Frith
frame. We will need the following lemma.

\begin{lemma}\label{densesym}
  If $(L,B)$ and $(M,C)$ are symmetric Frith frames, then any dense
  extremal epimorphism $h:(L,B)\ra (M,C)$ restricts and co-restricts
  to a Boolean algebra isomorphism $h': B \to C$.
\end{lemma}
\begin{proof}
  Suppose that $(L,B)$ and $(M,C)$ are symmetric Frith frames, and let
  $h:(L,B)\ra (M,C)$ be a dense extremal epimorphism. By
  Proposition~\ref{p:21}, we have $h[B] = C$. So, we only need to show
  that the restriction of $h$ to $B$ is injective. Let $b_1,b_2\in B$
  and suppose that $h(b_1)\leq h(b_2)$. This implies that $h(b_1)\we
  h(b_2)^*=0$ and, since frame morphisms preserve complemented pairs,
  we have $h(b_1\we b_2^*)=0$. By density, we may then conclude that
  $b_1\we b_2^*=0$ and, since $b_2$ is complemented, this implies
  $b_1\leq b_2$.
\end{proof}
Let $\overline{c}_{(L, S)}: \fsym(\idl(S), S) \twoheadrightarrow
\fsym(L, S)$ denote the symmetric reflection of $c_{(L, S)}$, that is,
$\overline{c}_{(L, S)}: (\cC_S\idl(S), \overline{S})
\twoheadrightarrow (\cC_SL, \overline{S})$ is defined by
$\overline{c}_{(L, S)}(\nabla_s) = \nabla_s$ and $\overline{c}_{(L,
  S)}(\Delta_s) = \Delta_s$, for $s \in S$.
\begin{proposition}\label{p:6}
  Let $(L, S)$, $(M, C)$ be Frith frames with $(M, C)$ symmetric and
  let $h:(M, C) \to (\cC_SL, \overline{S})$ be a dense extremal
  epimorphism. Then, there exists a unique morphism $g: (\cC_S\idl(S),
  \overline{S}) \to (M, C)$ making the following diagram commute:
  \begin{center}
    \begin{tikzpicture}[node distance = 20mm]
      \node (idl) {$(\cC_S\idl(S), \overline{S})$}; \node[right of =
      idl, xshift = 20mm] (M) {$(M, C)$}; \node[below of = M] (L)
      {$(\cC_SL, \overline{S})$};
      \draw[dashed, ->] (idl) to node[above] {$g$} (M); \draw[->>]
      (idl) to node[below] {$\overline c_{(L, S)}$} (L); \draw[->>]
      (M) to node[right] {$h$} (L);
    \end{tikzpicture}
  \end{center}
  Moreover, $g$ is a dense extremal epimorphism.
\end{proposition}
\begin{proof}
  Since $\overline{S}$ is join-dense in $\cC_S\idl(S)$, frame
  homomorphisms preserve complemented elements, and $C$ is closed
  under taking complements, if $g$ exists, then its underlying frame
  homomorphism is the unique extension of a lattice homomorphism $g':
  S \to C$ such that $h\circ g'(s) = \nabla_s$ for every $s \in S$. We
  first argue that the morphism $g'$ exists. Since $h$ is a dense
  extremal epimorphism, by Lemma~\ref{densesym}, it restricts and
  co-restricts to a lattice isomorphism $h':C\ra \overline{S}$. Let
  $g'$ be the restriction to~$S$ of the inverse of~$h'$.  Clearly,
  $g'$ satisfies $h \circ g'(s) = \nabla_s$ for every $s \in S$. The
  desired extension of $g'$ exists, too: we may first extend $g'$ to a
  frame homomorphism $\idl(S) \to M$, by Corollary~\ref{c:7}, and then
  to a frame homomorphism $g:\cC_S\idl(S) \to M$, by
  Proposition~\ref{p:10}.
  
  It remains to show that $g$ is a dense extremal epimorphism.
  Observe that $g$ suitably restricted and co-restricted is the
  inverse of $h':C\ra \overline{S}$. Therefore $g[\overline{S}] = C$
  and $g$ is an extremal epimorphism. Finally, since $\cC_S\idl(S)$ is
  generated, as a frame, by $\overline{S}$, $g$ is dense provided
  $\nabla_{s_1} \cap \Delta_{s_2} = 0$ for every $s_1, s_2 \in S$
  satisfying $g(\nabla_{s_1} \cap \Delta_{s_2}) = 0$, and this also
  follows from having that $g$ restricts and co-restricts to a lattice
  isomorphism $\overline{S} \to C$.
\end{proof}

\begin{corollary}\label{c:4}
  If $(L, S)$ is a Frith frame, then $(\idl(S), S)$ is complete and,
  therefore, $c_{(L, S)}: (\idl(S), S) \twoheadrightarrow (L, S)$ is a
  completion of~$(L, S)$.
\end{corollary}
\begin{proof}
  By definition, $(\idl(S), S)$ is complete if and only if so is
  $\fsym(\idl(S),S) = (\cC_S\idl(S), \overline{S})$. Let $h:(M, C)
  \twoheadrightarrow (\cC_S\idl(S), \overline{S})$ be a dense extremal
  epimorphism, with $(M, C)$ symmetric. Then, Proposition~\ref{p:6}
  applied to $h$ gives the existence of a dense extremal epimorphism
  $g: (\cC_S\idl(S),\overline{S}) \twoheadrightarrow (M, C)$
  satisfying $h \circ g = \overline{c}_{(\idl(S), S)}$.  Since,
  $\overline{c}_{(\idl(S), S)}$ is the identity function, $h$ is one-to-one, thus an isomorphism. Thus, $\fsym(\idl(S),S) =
  (\cC_S\idl(S), \overline{S})$ is complete as required.
\end{proof}

We may now state following pointfree analogue of
\cite[Theorem~4.1]{pin17}, which is a straightforward consequence
of~\eqref{eq:4} and of Corollary~\ref{c:4}.
\begin{theorem}\label{manycharacterizations}
For a Frith frame $(L,S)$ the following are equivalent.
\begin{enumerate}
    \item\label{item:3} The Frith frame $(L,S)$ is complete.
    \item\label{item:4} The Frith frame $(L,S)$ is coherent.
    \item\label{item:5} The Frith frame $\fsym(L,S)$ is coherent.
    \item\label{item:6} The Frith frame $\fsym(L,S)$ is compact.
\end{enumerate}
\end{theorem}

We finish this section by showing that completions are unique, up to
isomorphism. 
\begin{proposition}\label{p:5} Let $(L, S)$ be a Frith frame. Then,
  for every morphism $h:(M, T) \to (L, S)$ with $(M, T)$ complete,
  there exists a unique morphism $\widehat h: (M, T) \to (\idl(S), S)$
  such that the following diagram commutes:
   \begin{center}
    \begin{tikzpicture}[node distance = 20mm]
      \node (M) {$(M, T)$}; \node[right of = idl, xshift = 20mm] (idl)
      {$(\idl(S), S)$}; \node[below of = idl] (L) {$(L, S)$};
      \draw[dashed, ->] (M) to node[above] {$\widehat h$} (idl);
      \draw[->] (M) to node[below, xshift = -7pt] {$h$} (L);
      \draw[->>] (idl) to node[right] {${c}_{(L, S)}$} (L);
    \end{tikzpicture}
  \end{center}
  Moreover, if $h$ is dense (respectively, an extremal epimorphism)
  then $\widehat h$ is also dense (respectively, an extremal
  epimorphism).
\end{proposition}
\begin{proof} 
  Since $T$ is join-dense in~$M$, if such a homomorphism $\widehat h$
  exists, then it is completely determined by its restriction
  to~$T$. In particular, $\widehat h $ is unique because we must have
  $h(t) = c_{(L, S)}\circ \widehat h(t) = \widehat h (t)$, for every
  $t \in T$, that is, $\widehat h$ must be an extension of the
  restriction and co-restriction $h': T \to S$ of~$h$. By
  Corollary~\ref{c:7}, $h'$ uniquely extends to a frame homomorphism
  $\widehat h: \idl(T) \to \idl(S)$. Since, by
  Theorem~\ref{manycharacterizations}, $(M, T)$ is coherent and thus,
  by Proposition~\ref{p:17}, $M \cong \idl(T)$, it follows that
  $\widehat h$ is the required homomorphism.

  Now, suppose that $h$ is dense. Since $T$ is join-dense in $M$,
  $\widehat h$ is dense provided, for every $t \in T$, we have $t = 0$
  whenever $\widehat h(t) = 0$. This is indeed the case because
  $\widehat h(t) = h(t)$ for every $t \in T$. Finally, by the same
  reason, we also have that~$\widehat h$ is an extremal epimorphism if
  so is~$h$.
\end{proof}

\begin{corollary}\label{c:9}
  Each Frith frame has a unique, up to isomorphism, completion.
\end{corollary}
\begin{proof}
  By Corollary~\ref{c:4} we already know that every Frith frame $(L,
  S)$ has a completion $c_{(L, S)}: (\idl(S), S) \twoheadrightarrow
  (L, S)$. Let $c: (M, T) \twoheadrightarrow (L, S)$ be another
  completion of $(L, S)$. Since $c$ is a dense extremal epimorphism,
  by Proposition~\ref{p:5}, there exists a dense extremal
  epimorphism~$\widehat c:(M, T) \twoheadrightarrow (\idl(S), S)$
  satisfying $c_{(L, S)} \circ \widehat c = c$. Since $(M, T)$ is
  complete, $\widehat c$ has to be an isomorphism and so, the two
  completions of $(L, S)$ are isomorphic.
\end{proof}

%%%%%%%%%%%%%%%%%%%%%%%%%%%%%%%%%%%%%%%%%%%%%%%%%%%%%%%%%%%%%%%%%%%%%%
\subsection{Cauchy maps}\label{sec:9}

In this section we will show an analogue of Theorem~\ref{t:6}. We
start by proving some properties of a Cauchy map $\phi: (L, \cE) \to
M$, in the case where $\cE = \cE_B$, for some Boolean subalgebra $B
\subseteq L$, that is, $(L, B)$ is a symmetric Frith frame (recall
Theorem~\ref{t:1}) and $E(L, B) = (L, \cE_B)$.

\begin{lemma}\label{l:15}
  Let $(L, B)$ be a symmetric Frith frame, $M$ be a frame, and $\phi:
  (L, \cE_B) \to M$ be a Cauchy map. Then, the following statements
  hold:
  \begin{itemize}
  \item $\phi$ restricts to a lattice homomorphism with domain $B$,
  \item for every $a\in L$, the equality $\phi(a) = \bigvee \{\phi(b)
    \mid b \in B, \ b \leq a\}$ holds,
  \item for every $b \in B$, $\phi(b) \vee \phi(b)^* = 1$.
  \end{itemize}
\end{lemma}
\begin{proof}
  Let $\phi: (L, \cE_B) \to M$ be a Cauchy map, that is, $\phi$
  satisfies conditions~\ref{item:20}, \ref{item:1}, and \ref{item:2}
  of Definition~\ref{sec:cm}. We start by observing that, for every $b
  \in B$ we have $\phi(b^*) = \phi(b)^*$. Indeed, since $(a, a) \in
  E_b$ if and only if $a \leq b$ or $a \leq b^*$, and,
  by~\ref{item:20}, $\phi$ is order-preserving, by~\ref{item:2}, we
  have $1 = \phi(b) \vee \phi(b^*)$. Since, by~\ref{item:20},
  $\phi(b)\wedge \phi(b^*) = \phi(b \wedge b^*) = \phi(0) = 0$, it
  follows that $\phi(b)$ is complemented with complement $\phi(b^*)$,
  that is, $\phi(b^*) = \phi(b)^*$. In particular, we also have $1 =
  \phi(b) \vee \phi(b)^*$, which proves the third statement. Now, we
  let $b_1, b_2 \in B$. By~\ref{item:20}, the first statement holds
  provided $\phi(b_1 \vee b_2) \leq \phi(b_1) \vee \phi(b_2)$. Since
  each $\phi(b_i)$ is complemented, by~\eqref{eq:5}, this is
  equivalent to the equality $\phi(b_1 \vee b_2) \wedge
  \phi(b_1)^*\wedge \phi(b_2)^* = 0$, which follows from~\ref{item:20}
  together with the equality $\phi(b_i^*) = \phi(b_i)^*$ already
  proved. Finally, let us show that the second statement is valid. We
  fix some $a \in L$. Since $B$ is a Boolean algebra, by
  Lemma~\ref{l:17}, for every $x \in L$ satisfying $x \luni_1 a$ or $x
  \luni_2 a$, there is some $b_x \in B$ such that $x \leq b_x \leq a$.
  Then, using~\ref{item:1} and the fact that $\phi$ is
  order-preserving, we may derive that
  \begin{align*}
    \phi(a)
    & \leq \bigvee \{\phi(x) \mid x \in L, \ x \luni_1 a \text{ or }x \luni_2
      a\} \leq \bigvee \{\phi(b_x) \mid x \in L, \ x \luni_1 a \text{ or
      }x \luni_2 a\}
    \\ & \leq \bigvee \{\phi (b) \mid b \in B, \ b \leq
         a\} \leq \phi(a).\popQED\qed
  \end{align*}
\end{proof}

It is then natural to consider the following definition of
\emph{Cauchy map}.
\begin{definition}\label{sec:cm1}   
  Let $(L, S)$ be a Frith frame and $M$ any frame. A \emph{Cauchy map}
  $\phi: (L, S) \to M$ is a function $\phi: L \to M$ such that
  \begin{enumerate}[label = (C.\arabic*)]
  \item\label{item:C1} $\phi$ restricts to a lattice homomorphism with
    domain~$S$,
  \item\label{item:C2} for every $a \in L$, the equality $\phi(a) =
    \bigvee \{\phi (s) \mid s \in S, \ s \leq a\}$ holds,
  \item\label{item:C3} for every $s \in S$, $\phi(s) \vee \phi(s)^* =
    1$.
  \end{enumerate}
\end{definition}
By Lemma~\ref{l:15}, we have that in the case where $(L, B)$ is a
symmetric Frith frame, if $\phi: (L, \cE_B) \to M$ is a Cauchy map in
the sense of Definition~\ref{sec:cm}, then $\phi: (L, B) \to M$ is a
Cauchy map in the sense of Definition~\ref{sec:cm1}. We will now show
that the converse is also true, so that our definition of Cauchy map
agrees with the classical one for transitive and totally bounded
uniform frames (recall Corollary~\ref{c:5}).
\begin{proposition}
  Let $(L, B)$ be a symmetric Frith frame, $M$ a frame, and $\phi: L
  \to M$ a function. Then, $\phi$ defines a Cauchy map $\phi: (L,
  \cE_B) \to M$ if and only if it defines a Cauchy map \mbox{$\phi:
    (L, B) \to M$}.
\end{proposition}
\begin{proof}
  The forward implication is the content of Lemma~\ref{l:15}.
  Conversely, let $\phi: (L, B) \to M$ be a Cauchy map, where $(L, B)$
  is a symmetric Frith frame. We first argue that $\phi$ is a bounded
  meet homomorphism. Since $\phi|_B$ is a lattice homomorphism, we
  have $\phi(0) = 0$ and $\phi(1) = 1$. Let $a_1, a_2 \in L$. Then, we
  may compute
  \begin{align*}
    \phi(a_1 \wedge a_2)
    & \just {\ref{item:C2}} = \bigvee \{\phi(b) \mid b \in B, \ b
      \leq a_1 \wedge a_2\}\just {\ref{item:C1}} =
      \bigvee \{\phi(b_1) \wedge \phi(b_2)\mid b_1, b_2 \in B, \ b_1
      \leq a_1, \ b_2 \leq a_2\}
    \\ &  \hspace{2mm}= (\bigvee \{\phi(b_1) \mid b_1 \in B, \ b_1 \leq a_1\})
         \wedge (\bigvee \{\phi(b_2) \mid b_2 \in B, \ b_2 \leq a_2\})
         \just {\ref{item:C2}} =\phi(a_1) \wedge \phi(a_2).
  \end{align*}
  Thus, $\phi$ also preserves binary meets and we
  have~\ref{item:20}. Now, using~\ref{item:C2}, in order to
  show~\ref{item:1}, it suffices to observe that, for every $b \in B$,
  we have $b \luni_1 a$ whenever $b \leq a$. Indeed, that is a
  consequence of the inclusion $E_b \circ (b \oplus b) \subseteq (a
  \oplus a)$ for every $b \leq a$. It remains to
  show~\ref{item:2}. Let $E \in \cE_B$, say $E \supseteq \bigcap_{i =
    1}^n E_{b_i}$ for some $b_1, \dots, b_n \in
  B$. Using~\ref{item:C1} and the fact that each
  $b_i$ is complemented, we have
  \[1 = \phi (\bigwedge_{i = 1}^n (b_i \vee b_i^*)) = \bigwedge_{i =
      1}^n (\phi(b_i) \vee \phi(b_i^*)) = \bigvee \{\phi(b_P \wedge
    \overline{b}_P) \mid P \subseteq [n]\},\]
  where $[n] := \{1, \dots, n\}$ and, for $P \subseteq [n]$, we denote
  $b_P:= \bigwedge_{i \in P} b_i$ and $\overline{b}_P := \bigwedge_{i
    \notin P} b_i^*$.  Since, for every $P \subseteq [n]$, we have
  $(b_P \wedge \overline{b}_P, b_P \wedge \overline{b}_P) \in
  \bigcap_{i = 1}^n E_{b_i} \subseteq E$, it then follows that
  \[1 =\bigvee \{\phi(b_P \wedge \overline{b}_P) \mid P \subseteq
    [n]\} \leq \bigvee \{\phi(x) \mid (x, x) \in E\},\]
  which proves~\ref{item:2}.
\end{proof}

In what follows, we fix a Frith frame $(L, S)$ and a frame $M$. Recall
from the previous section that, for every Frith frame $(L, S)$, there
is a dense extremal epimorphism $c: \idl(S) \twoheadrightarrow L$
defined by $c(J) = \bigvee J$. Since this is a frame homomorphism, it
has a right adjoint $c_*: L \hookrightarrow \idl(S)$ which is
determined by the Galois connection
\[\forall a \in L, \ J \in \idl(S), \quad c(J) \leq a \iff J \subseteq
  c_*(a).\] In particular, for every $a \in L$, we have
\begin{equation}
  c_*(a) = \bigvee \{ J \in \idl(S)\mid \bigvee J \leq a\}.\label{eq:9}
\end{equation}
\begin{lemma}\label{l:6}
  For every $a \in L$, we have $c_*(a) = {\downarrow} a \cap S$. In
  particular, the map $c$ is a right inverse of $c_*$, that is, $c
  \circ c_* = \id$.
\end{lemma}
\begin{proof}
  Let $s \in S$. By~\eqref{eq:9}, we have that $s \in c_*(a)$ if and
  only if there are $J_1, \dots, J_n \in \idl(S)$, and $s_i \in J_i$
  (for $i = 1, \dots, n$) such that $\bigvee J_i \leq a$ and $s \leq
  s_1 \vee \dots \vee s_n$. Clearly, this holds if and only if $s \leq
  a$ and thus, we have $c_*(a) = {\downarrow} a \cap S$.  Since $S$ is
  join-dense in~$L$, it then follows that $c\circ c_*(a) =
  c({\downarrow} a \cap S) = \bigvee ({\downarrow} a \cap S) = a$.
\end{proof}

We now consider the function $\lambda: L \hookrightarrow \cC_S\idl(S)$
obtained by composing the injections $c_*: L \hookrightarrow \idl(S)$
and $\nabla: \idl(S) \hookrightarrow \cC_S\idl(S)$. Explicitly,
$\lambda$ sends the element $a \in L$ to the congruence
$\nabla_{{\downarrow a}\cap S}$.  Notice that~$\lambda$ defines a
Cauchy map $\lambda: (L, S) \to \cC_S\idl(S)$. Indeed, since $c_*(s) =
{\downarrow} s \cap S$ for every $s \in S$ and $\nabla$ is a frame
homomorphism, we have that $\lambda$ restricts to a lattice
homomorphism with domain~$S$, that is, \ref{item:C1} holds.  Moreover,
since ${\downarrow} a \cap S$ is the ideal generated by
$\bigcup\{{\downarrow}s \cap S \mid s \in S, \ s \leq a\}$, we have
\ref{item:C2}. Finally, by definition, each $\lambda(s) = \nabla_{s}$
is complemented in $\cC_S \idl(S)$, and so, we have~\ref{item:C3}.

\begin{theorem}\label{t:2}
  For every Cauchy map $\phi: (L, S) \to M$, there exists a frame
  homomorphism $g: \cC_S\idl(S) \to M$ such that the following
  diagram commutes:
  \begin{center}
    \begin{tikzpicture}[node distance = 0mm]
      \node (L) {$L$}; \node[right of = idl] (idl) {$\cC_S\idl(S)$};
      \node[below of = idl, yshift = -15mm] (M) {$M$};
      \draw[dashed, ->] (idl) to node[right] {$g$} (M);
      \draw[->] (L) to node[below] {$\phi$} (M);
      \draw[->] (L) to node[above] {$\lambda$} (idl);
    \end{tikzpicture}
  \end{center}
\end{theorem}
\begin{proof}
  Since $\cC_S\idl(S)$ is generated, as a frame, by the set of
  congruences $\{\nabla_s, \Delta_s \mid s \in S\}$ and frame
  homomorphisms preserve pairs of complemented elements, if $g$
  exists, then it is completely determined by its restriction to
  $\{\nabla_s \mid s \in S\}$, which must satisfy $g(\nabla_s) =
  g\circ \lambda(s) = \phi(s)$, for every $s \in S$. By~\ref{item:C1},
  $\phi$ restricts to a lattice homomorphism $\phi|_S: S \to M$ and,
  by Corollary~\ref{c:7}, $\phi|_S$ uniquely extends to frame
  homomorphism $\widehat{\phi|_S}: \idl(S) \to M$. Since,
  by~\ref{item:C3}, each $\phi(s)$, with $s \in S$, is complemented
  in~$M$, we may then use Proposition~\ref{p:10} to derive the
  existence of a unique frame homomorphism $g: \cC_S\idl(S) \to M$
  satisfying $g(\nabla_s) = \phi(s)$, for every $s \in S$. To see that
  $g$ makes the diagram commute, we may use \ref{item:C2} for
  $\lambda$ and for $\phi$ and the fact that $g$ is a frame
  homomorphism to compute
  \[g\circ \lambda(a) = g(\bigvee \{\lambda(s) \mid s \in S, \ s \leq a\}) =
    \bigvee \{g\circ \lambda(s) \mid s \in S, \ s \leq a\} = \bigvee
    \{\phi(s) \mid s \in S, \ s \leq a\} = \phi(a),\]
  for every $a \in L$.
\end{proof}

\begin{theorem}\label{t:7}
  A Frith frame $(L, S)$ is complete if and only if every Cauchy map
  $(L, S) \to M$ is a frame homomorphism.
\end{theorem}
\begin{proof}
  If $(L, S)$ is complete, then $c$ is an isomorphism and thus, $c_*$,
  hence $\lambda$, is a frame homomorphism. Since by Theorem~\ref{t:2}
  every Cauchy map factors through~$\lambda$ via a frame homomorphism,
  it follows that every Cauchy map is itself a frame homomorphism.

  Conversely, if every Cauchy map is a frame homomorphism then
  $\lambda$ is a frame homomorphism. In particular, $\lambda$ induces
  a morphism of Frith frames $\lambda: (L, S) \to (\cC_S\idl(S),
  \overline{S})$. On the other hand, by Corollary~\ref{c:4} and by
  definition of complete Frith frame, we have that $(\cC_S\idl(S),
  \overline{S}) = \fsym(\idl(S), S)$ is complete. Therefore, $(L, S)$
  is complete provided the symmetric reflection $\overline{\lambda}:
  (\cC_SL, \overline{S}) \to (\cC_S\idl(S), \overline{S})$
  of~$\lambda$ is a dense extremal epimorphism. That is the case
  because, for every $s \in S$, we have $\overline{\lambda}(\nabla_s)
  = \nabla_s$ and $\overline{\lambda}(\Delta_s) = \Delta_s$.
\end{proof}
%%%%%%%%%%%%%%%%%%%%%%%%%%%%%%%%%%%%%%%%%%%%%%%%%%%%%%%%%%%%%%%%%%%%%%

\section{Notes on the existing literature}\label{sec:10}

  Join-dense subsets of a frame $L$ are usually called \emph{bases for
    $L$}. Thus, a Frith frame $(L, S)$ is nothing but a frame equipped
  with a bounded sublattice base~$S$. Frames with special bases had
  already been considered in the literature, see for instance,
  \cite{MR1077908} for frames with \emph{normal bases},
  \cite{MR1809213} for frames with \emph{cozero bases}, and
  \cite{doi:10.2989/16073606.2012.697268} for frames with Wallman
  bases. In all these works, it is shown that a base of the
  appropriate type for a frame $L$ induces a \emph{strong inclusion}
  on $L$. In our case, a non-symmetric version of this holds: given a
  Frith frame $(L, S)$, the relation $\luni_S \subseteq L \times L$
  defined by
  \[a \luni_S b \iff \exists s \in S \colon a \leq s \leq b\]
  is a \emph{proximity} on $L$ (in the sense of~\cite{MR3185523}) that
  satisfies the following additional property: if $a \luni_S b$ then
  there is $c \in L$ such that $a \luni_S c \luni_S c \luni_S
  b$. Actually, every such proximity~$\luni$ is defined by a Frith
  frame: just take $S = \{ a\in L \mid a \luni a\}$. In the case where $S$
  is a Boolean algebra, the relation $\luni_S$ is a strong inclusion
  on~$L$, and $S$ is named a \emph{c-base of $L$} in~\cite{MR1809213}.
  While \emph{strong inclusions} on $L$ are known to be in a bijective
  correspondence with \emph{compactifications} of
  $L$~\cite{MR1124796}, \emph{proximities} on $L$ are in a bijective
  correspondence with \emph{stable
    compactifications}~\cite{MR3185523}. It is not hard to see that
  the stable compactification associated with the proximity $\luni_S$
  is precisely the underlying frame homomorphism of~$c_{(L,S)}$
  defined in~\eqref{eq:3}.

  Despite this parallel, there are major differences in the spirit of
  our work when compared to existing literature.  On the one hand, to
  the best of our knowledge, a functorial approach to frames with
  bases, where these are seen as the object part of a suitable
  category, has not been undertaken. On the other hand, we regard a
  Frith frame as a representation of a quasi-uniformity on a certain
  frame of congruences, which is used for studying the latter.

  Finally, we note that there are well-known connections between
  strong inclusions and uniformities, namely, Frith~\cite{Frith86}
  proved that frames equipped with strong inclusions are part of a
  category isomorphic to that of totally bounded uniform frames. In
  the case where $(L, B)$ is a symmetric Frith frame, the frames $L$
  and $\cC_B L$ are isomorphic and our (transitive and) totally
  bounded uniformity $\cE_B$ on $L$ is precisely the uniformity
  defined by $\luni_B$ under Frith's correspondence. Accordingly, it
  would be worth investigating and characterize those totally bounded
  uniform frames arising from a frame with a normal/Wallman/cozero
  base. In the vein of~\cite{MR1809213}, it would also be interesting
  to describe those frames $L$ that are not generated by any proper
  bounded sublattice, being $\omega + 1$ an example of such. A
  solution to the symmetric version of this question is the content
  of~\cite[Proposition~7]{MR1809213}.

%%%%%%%%%%%%%%%%%%%%%%%%%%%%%%%%%%%%%%%%%%%%%%%%%%%%%%%%%%%%%%%%%%%%%% 
\section*{Acknowledgments}
The authors would like to thank the anonymous referee for the careful
reading of the paper and for the useful suggestions that helped to
improve the presentation of our work.

\bibliographystyle{acm}
% \bibliography{../bib.bib}

\end{document}